\documentclass{article}

\usepackage{latexsym,enumerate}
\usepackage{amsmath,amsthm}
\usepackage{amssymb,amsfonts}
\usepackage{mathrsfs}
\usepackage{epsfig,color}

\def\N{\mathbb{N}}

\def\R{\mathbb{R}}

\def\pa{\partial}
\def\RR{\mathbb{R}}
\def\eps{\varepsilon}

\def\hU{\widehat U}
\def\hux{\widehat {u_x}}
\def\huy{\widehat {u_y}}

\DeclareMathOperator{\divv}{div}
\DeclareMathOperator{\supp}{supp}
\DeclareMathOperator{\trace}{trace}

\def\argsinh{\operatorname{argsinh}}

\def\ds{\displaystyle}

\newtheorem{theorem}{Theorem}[section]
\newtheorem{lemma}[theorem]{Lemma}

\newtheorem{proposition}[theorem]{Proposition}

\theoremstyle{remark}
\newtheorem{remark}[theorem]{Remark}

\title{Self-similar solutions for a fractional thin film equation
  governing hydraulic fractures}

 \author{C. Imbert\footnote{CNRS, UMR 7580, Universit\'e Paris-Est
     Cr\'eteil, 61 avenue du G\'en\'eral de Gaulle, 94 010 Cr\'eteil
     cedex, France} and A. Mellet\footnote{Department of
     Mathematics, University of
     Maryland, College Park, MD 20742, USA}{ }\footnote{Fondation Sciences Math\'ematiques de Paris, 11 rue Pierre et Marie Curie, 75231 PARIS Cedex 05}}
\date{}

\begin{document}

\maketitle

\begin{abstract}
In this paper, self-similar solutions for a fractional thin film
equation governing hydraulic fractures are constructed.  One of the
boundary conditions, which accounts for the energy required to break
the rock, involves the toughness coefficient $K\geq 0$.
Mathematically, this condition plays the same role as the contact
angle condition in the thin film equation.  We consider two
situations: The zero toughness ($K=0$) and the finite toughness
$K\in(0,\infty)$ cases.  In the first case, we prove the existence of
self-similar solutions with constant mass. In the second case, we prove
that for all $K>0$ there exists an injection rate for the fluid such
that self-similar solutions exist.
\end{abstract}



\tableofcontents

\section{Introduction}
\subsection{The model}

The following third order degenerate parabolic equation arises in the
modeling of hydraulic fractures:
\begin{equation}\label{eq:tf0}
\pa_t u +\pa_x (u^3 \pa_x I(u)) = S
\end{equation}
where the operator $I$ denotes the square root of the Laplace
operator: 
\[I(u) = -(-\Delta )^{\frac 1 2} u.\]
This equation can be seen as a fractional version of the thin film
equation (which corresponds to $I(u)=\Delta u$). It is also
reminiscent of the porous media equation, which corresponds to
$I(u)=-u$.

In the context of hydraulic fractures, the unknown $u(x,t)$ represents
the opening of a rock fracture which is propagated in an elastic
material due to the pressure exerted by a viscous fluid which fills
the fracture.  Such fractures occur naturally, for instance in
volcanic dikes where magma causes fracture propagation below the
surface of the Earth, or can be deliberately propagated in oil or gas
reservoirs to increase production.  The term $S$ in the right hand
side of the equation is a source term which models the injection of
fluid into the fracture.  It is usually assumed to be $0$ or of the
form $h(t)\delta(x)$ (corresponding at the injection of fluid into the
fracture at a rate $h(t)$ through a pipe located at $x=0$).

There is a significant amount of work involving the mathematical
modeling of hydraulic fractures (see for instance Barenblatt
\cite{barenblatt} and references therein).  The model that we consider
in our paper, which corresponds to a very simple fracture geometry,
was developed independently by Geertsma and De Klerk \cite{DK} and
Kristianovich and Zheltov \cite{ZK}. Note that the profile of the self-similar
solution of the porous medium equation exhibited independently by
Zeldovitch and Kompaneets \cite{zeld} and Barenblatt \cite{bare} is a
stationary solution of \eqref{eq:tf0}.  Spence
and Sharp \cite{SS85} initiated the work on self-similar solutions and
formal asymptotic analysis of the solutions of
(\ref{eq:tf0}) near the tip of the fracture (i.e. the boundary of the
support of $u$).  There is now an abundant literature that has
extended this formal analysis to various regimes (see for instance
\cite{ad}, \cite{AP08}, \cite{MKP} and references therein).  Several
numerical methods have also been developed for this model (see in
particular Peirce et al. \cite{PD08}, \cite{PD09}, \cite{PS05} and
\cite{PS05b}).

In a recent paper \cite{im}, we established the existence of weak
solution to this equation in a bounded interval. To our knowledge,
this was the first rigorous existence result for this equation.  In
fact, in that paper, we considered the more general equation
\[\pa_t u +\pa_x (u^n \pa_x I(u)) = 0\]
for any $n\geq 1$.  Indeed, as shown in \cite{im}, the particular
value $n=3$ in \eqref{eq:tf0} follows from the choice of no-slip
Navier boundary conditions for the fluid in contact with the rock.
However, as for the thin film equation, other values of $n$ (namely
$n=1$ and $n=2$) are also of interest when other types of fluid
boundary conditions are considered in the derivation of the equation
(see Subsection~\ref{sec:phys} below and \cite{im}).  Mathematically,
the properties of the solutions depend strongly on the value of the
parameter $n$, as is the case for the thin film equation.  In fact,
the results of \cite{im} show that many aspects of the mathematical
analysis of \eqref{eq:tf0} are similar to the theory of the thin film
equation; however, the fact that $I$ is a non-local operator
introduces many new difficulties to the problem.  It was also pointed
out in \cite{im} that the value $n=4$ is critical for this equation,
in the same way that the value $n=3$ is known to be critical for the
thin film equation.  As we will see, the main results in our paper
will indeed require that $n<4$.

\medskip

Before going any further, we need to determine the appropriate
boundary conditions.  Equation \eqref{eq:tf0} is satisfied within the
fracture, that is in the region $\{u>0\}$ (note that $u$ has to be
defined in whole of $\R$ so that the non-local square root of the
Laplacian can be defined).  At the tip of the fracture, that is on
$\pa\{u>0\}$, it is natural to assume a null flux boundary condition
(no leak of fluid through the rock):
\[u^n \partial_x I (u) = 0\quad  \text{ on } \pa\{u>0\}\]
(models involving leak at the tip of the fracture are also of interest
for applications, but will not be discussed in this paper).  Together
with the fact that $u=0$ in $\R^n\setminus \{u>0\}$, this gives us two
boundary conditions.  Since we are dealing with a free boundary
problem of order three, these two conditions are not enough to have a
well posed problem.  The missing condition takes into account the
energy required to break the rock and takes the form (see for instance \cite{MKP}):
\begin{equation}\label{eq:bc2}
u(t,x)= K\sqrt{|x-x_0|} + o (|x-x_0|^{1/2}) \quad \text{ as } x\to x_0
\end{equation}
for all $x_0\in\pa\{u(t,\cdot)>0\}$ where the coefficient $K$ is
related to the toughness of the rock and is assumed to be known.  From
a mathematical point of view, we note that condition \eqref{eq:bc2}
plays the same role as the contact angle condition for the thin film
equation.

The particular case $K=0$ is mathematically interesting (it
corresponds to the ``zero contact angle'' condition -- or complete
wetting regime -- often studied in the thin film literature).  In the
framework of hydraulic fracture, this {\bf zero toughness} condition
can be interpreted as modeling the expansion of a fracture in a
pre-cracked rock.

\medskip

Note that in \cite{im}, we did not include a free boundary condition,
and instead considered that Equation \eqref{eq:tf0} was satisfied
everywhere.  The solutions that we constructed there belonged to
$L^2_t(H^{3/2}_x)$ and thus satisfied $u(t,\cdot)\in C^{\alpha}$ for
a.e. $t>0$ for all $\alpha<1$. In particular, compactly supported
solutions would satisfy \eqref{eq:bc2} on the boundary of their
support (or tip of the fracture) with $K=0$.
In other words, the solutions constructed in \cite{im} correspond to the zero toughness (or pre-cracked) regime.
In the
present paper, we consider the full free boundary problem 
with $K>0$
and we will
prove the existence of self-similar solutions in both the zero
toughness case (without injection of fluid) and the non-zero toughness case (with specific injection rate).  
This is thus the first
rigorous existence results for solutions satisfying the free boundary
condition \eqref{eq:bc2} with $K> 0$.
We also rigorously investigate the behavior of the solution at the tip of the fracture.

\medskip

In the case of the thin film equation ($I(u)=\Delta u$), the existence
of self-similar solutions has been proved in the zero contact angle
case (which corresponds to the case $K=0$ here) in particular by
Bernis, Peletier and Williams \cite{BPW} in dimension $1$ and by
Ferreira and Bernis \cite{bf} in dimension greater than $2$.  It is
worth noticing that while our result concerns only the dimension $1$,
the proofs will be somewhat more similar to the higher dimensional
case for the thin film equation.

\subsection{Main results}
To summarize the introduction above, the equation under consideration in this paper is the following:
\begin{equation}\label{eq:ftf}
\partial_t u + \partial_x (u^n \partial_x I (u)) =h(t)\delta, \quad
t>0, \mbox{ in }  \{u>0\},
\end{equation}
where $n \ge 1$, together with the boundary conditions 
\begin{equation}\label{eq:bc}
u^n \partial_x I (u) = 0\quad  \mbox{ on } \pa\{u>0\}
\end{equation}
and 
\begin{equation}\label{eq:tough}
  u(t,x)= K\sqrt{|x-x_0|} + o \left(\sqrt{|x-x_0|}\right) \quad \text{ as } x\sim x_0
\end{equation}
for all $x_0\in \pa\{u(t,\cdot)>0\}$.

\medskip

The two main parameters are the function $h(t)$, which corresponds to
the injection rate of the fluid into the fracture, and the constant
$K$, which describes the toughness of the rock.  Note that when $h=0$
(no injection of fluid) and $K\neq 0$, then
\eqref{eq:ftf}-\eqref{eq:bc}-\eqref{eq:tough} has a stationary
solution supported in $(-1,1)$ given by
\[ V(x)= \frac{K}{\sqrt 2} \sqrt{(1-x^2)_+}.\]
(this is checked easily using the fact that $ I (\sqrt{(1-x^2)_+})=-
\frac{2}{\sqrt{\pi}}$).  Clearly, the function $\sqrt a\, V(x/a)$ is
also a stationary solution supported in $(-a,a)$ for any $a>0$.

\medskip

The goal of this paper is to prove the existence of another type of
particular solutions of \eqref{eq:ftf}-\eqref{eq:bc}-\eqref{eq:tough}:
compactly supported self-similar solutions.  More precisely, we are
looking for solutions of the form
\begin{equation}\label{eq:profile} 
u(t,x) = t^{-\alpha} U (t^{-\beta} x) 
\end{equation}
for some \emph{profile function} $U$ which is even and supported in an
interval $[-a,a]$ for some $a>0$.

Inserting \eqref{eq:profile} into \eqref{eq:ftf}, we find that $U$
must solve
\begin{multline*}
-\alpha U(t^{-\beta} x) - \beta t^{-\beta} x U'(t^{-\beta} x) 
+ t^{-n\alpha +1-3\beta} (U^n I(U)')'(t^{-\beta} x)  \\ = t^{1+\alpha}h(t)
t^{-\beta} \delta(t^{-\beta} x)
\end{multline*}
(using the fact that $t^{-\beta} \delta(t^{-\beta} x)=\delta(x)$).
So we must take  $\alpha$ and $\beta$ such that
\begin{equation}\label{eq:ab} 
1-3\beta = n \alpha
\end{equation}
and the injection rate $h(t)$ given by
\begin{equation}\label{eq:h} 
h(t)=\lambda t^{-\alpha -1+\beta}
\end{equation}
for some constant $\lambda \in \R$.  Then the profile $y\mapsto U(y)$
is solution to the equation
\begin{equation}\label{eq:UU}
-\alpha U  - \beta y U'  +(U^n I(U)')'  =\lambda \delta \qquad \mbox{ in } (-a,a). 
\end{equation}

The profile function $U$ must also satisfy appropriate boundary
conditions.  Clearly, if $U$ satisfies
\begin{equation}\label{eq:bcU}
U^n  I (U)' = 0\quad  \mbox{ on } \pa\{U>0\}
\end{equation}
then $u$ will satisfy \eqref{eq:bc}.  The boundary condition
\eqref{eq:tough}, however, is more delicate. Indeed, we notice that if
$U$ satisfies
\[ U(x) = K \sqrt{|x-a|} + o\left(\sqrt{|x-a|}\right),\]
then the function $u(t,x)$ defined by \eqref{eq:profile} satisfies
\begin{equation}\label{eq:tout}
 u(t,x) = K t^{-\alpha-\beta/2} |x-a(t)|^{1/2} + o (|x-a(t)|^{1/2})
\end{equation}
with $a(t)=t^{\beta} a \in \pa\{u(t,\cdot)>0\}$.  So a self-similar
solution $u(t,x)$ can only satisfy the free boundary
condition~\eqref{eq:tough} with given, time independent, toughness
coefficient $K$ if either $K=0$ (zero toughness) or if $\alpha = -
\frac{\beta}{2}$
  
We will thus construct two types of self-similar solutions:
\begin{itemize}
\item In the case where no fluid is injected ($h(t)=0$), we will show
  that there exist self-similar solutions satisfying \eqref{eq:tough}
  with $K=0$ (zero toughness case) and constant mass $m$ (in
  particular $\alpha=\beta$);
\item For given toughness coefficient $K>0$, we will show that there
  exists an injection rate $h(t)$ (of the form \eqref{eq:h}) such that
  there exists a self-similar solution satisfying \eqref{eq:tough} for
  all $t$.
\end{itemize}

More precisely, our main result is the following:
\begin{theorem}\label{thm:main}
Assume that $n\in [1,4)$.
\item[\upshape (i)] Assume that $K=0$ and $h(t)=0$. Then, for any $m>0$ there
  exists a self-similar solution of
  \eqref{eq:ftf}-\eqref{eq:bc}-\eqref{eq:tough} of the form
\[ u(t,x) = t ^{-\frac{1}{n+3} } U (t ^{-\frac{1}{n+3} } x)\]
satisfying $\int_\R u(t,x)\, dx = m$ for all $t>0$.
The profile function $x\mapsto U(x)$ is a non-negative, even function with $\supp \, U=[-a,a]$ for some $a>0$ (depending on $m$).
Furthermore, for all $t>0$, there exists a constant $C(t)>0$ such that  $u$ satisfies
\[ u(t,x)=
\left\{
\begin{array}{ll}
C(t)|x-x_0|^{\frac{3}{2}}+  \mathcal O (|x-x_0|^{\frac 2  n})  & \mbox{ if } n \in [1,\frac{4}{3})\medskip\\
C(t)|x-x_0|^{\frac32} \big|\ln  |x-x_0| \big|^{\frac34}+\mathcal O (|x-x_0|^{\frac 32}) & \mbox{ if } n= \frac{4}{3}\medskip\\
C(t)|x-x_0|^{\frac{2}{n}} +o (|x-x_0|^{\frac 2  n})& \mbox{ if } n \in  (\frac{4}{3},4) 
\end{array}
\right.\]
when $x\to x_0$,  for any $x_0 \in \pa\{u(t,\cdot)>0\}$.
\item[\upshape (ii)] For any $K> 0$ and for any $a>0$
there exists $\lambda>0$ such that equation \eqref{eq:ftf}-\eqref{eq:bc}-\eqref{eq:tough} 
has a self-similar solution when $h(t) = \lambda \, t^{\frac{n-3}{6-n}}$. This solution has the form
\[u(t,x) = t^{\frac{1}{6-n}} U(t^{-\frac{2}{6-n}} x)\]
where $U$ is a non-negative, even function with $\supp\, U=[-a,a]$. 
Furthermore, $u$ satisfies
\[u(t,x)= K\sqrt{|x-x_0|} + 
\left\{ 
\begin{array}{ll}
\mathcal O(|x-x_0|^{\frac{3}{2}}) & \mbox{ if } n \in  [1,2) \medskip \\
\mathcal O\left(|x-x_0|^{\frac32}  \ln\left(\frac{1}{|x-x_0|}\right) \right)& \mbox{ if } n=2\medskip\\
\mathcal O(|x-x_0|^{\frac{5-n}{2}}) & \mbox{ if } n \in (2,4)
\end{array}
\right.\]
when $x\to x_0$,  for any $x_0 \in \pa\{u(t,\cdot)>0\}$.
\end{theorem}

\begin{remark}\label{rem:injec-n3}
\item[\upshape 1.] Note that in the physical case, that is when $n=3$, we find
  $h(t)=\lambda$, so self-similar solutions in that case correspond to a
  constant injection rate.

\item[\upshape 2.] 
Note also that in the first part of the theorem ($K=0$), the
  self-similar solution satisfies
\[ \lim_{t\to 0^+} u(t,x) = m \delta \] 
in the sense of distributions. Such a solution is also sometimes called
a Source-type solution.  On the other hand, in the second part ($K\neq
0$), we clearly have
\[ \lim_{t\to 0^+} ||u(t,x)||_{L^\infty} = 0 .\]
\item[\upshape 3.] 
In the case $n=3$, $K>0$, we recover here known (formal) results concerning the rate of growth of hydraulic fractures (see \cite{MKP2,garagash,ad,ddlp}): The length of the fracture is proportional to $t^{2/3}$ and its width ($=u(t,0)$) is proportional to $t^{1/3}$. We also recover the following asymptotic at the tip of the fracture (see \cite{MKP2})
$$ u(t,x) = K\sqrt{|x-x_0|} + \mathcal O(|x-x_0|).$$
\item[\upshape 4.] 
In the second part of the theorem, we fix $K$ and $a$ (which is half the length of the support of $u$ at time $t=1$), and find the appropriate value of $\lambda$  for a solution to exist. It would be more satisfactory to show that a solution exists for all values of  $K>0$ and $\lambda>0$.
We will see in the next section that the constant $\lambda$ satisfies 
\[ \lambda = \frac3{6-n}  \int_{-a}^a U (x) dx.\]
Using this relation, we can then show that for a given $K$, we have $\lambda(a)\to 0$ as $a\to 0$ and $\lambda(a)\to \infty$ as $a\to \infty$.
It seems thus reasonable to expect that for all $K$ and for all $\lambda>0$, there exists a self similar 
solution of  \eqref{eq:ftf}-\eqref{eq:bc}-\eqref{eq:tough} (which is obtained for an appropriate choice of $a$). 
However, to prove this rigorously, one needs to show that the the function $a\mapsto \lambda(a)$ is continuous, and such result should typically follow from some uniqueness principle for $U$.

Unfortunately the question of the uniqueness of the self-similar solution for this problem, which is of independent interest, is notoriously hard to obtain for such non-linear  higher order equations. 
In \cite{bf}, Ferreira and Bernis prove the uniqueness of  self similar solutions for the thin film equation in the zero contact angle case. However, even in the zero toughness case, such a proof does not seem to extend to our case, mainly because of the nonlocal character of the fractional Laplacian. 
The question of the uniqueness of self similar solutions, both in the case $K=0$ and $K>0$ is thus left as an interesting and challenging open problem here.

\end{remark}
\medskip

In the next section, we will set up the equations to be solved by the
profile $U(x)$ in both cases of Theorem \ref{thm:main}. At the end of
that section (see Subection~\ref{sec:strategy} below), we describe the
general strategy to be used, which is reminiscent of the approach of
Bernis and Ferreira \cite{bf} for the thin film equation in dimension
greater than or equal to $2$.  In particular, this strategy relies on
an integral formulation and a fixed point argument, which requires a
detailed knowledge of the Green function associated to the operator
$u\mapsto I(u)'$.  The properties of this Green function are discussed
in Section \ref{sec:Green} which is the core of this paper. In
particular, very detailed results concerning the boundary behavior of
the solution of the equation $I(u)'=f$ are given in that section.
These results play a fundamental role in the proof of our main result,
which is given in Section \ref{sec:main}.

\subsection{Derivation of the model}\label{sec:phys}
As mentioned in the introduction, when $n=3$, Equation (\ref{eq:tf0}) was introduced to describe the propagation of an
impermeable KGD fracture (named after Kristianovich, Geertsma and De Klerk) driven by a viscous fluid in a uniform
elastic medium under condition of plane strain.  
We recall in this section the main steps of this derivation (see \cite{DK,ZK}).
Everything in this section can be found in the literature, and is recalled here for the reader's sake.  
We denote by $(x,y,z)$ the standard coordinates in $\R^3$, we consider
a fracture which is invariant with respect to one variable ($z$)
and symmetric with respect to another direction ($y$).  The
fracture can then be entirely described by its opening $u(x,t)$ in the
$y$ direction.
Since it assumes that the fracture is an infinite strip whose
cross-sections are in a state of plane strain, this model is only applicable
to rectangular planar fracture with large aspect ratio.

\paragraph{Lubrication approximation.}\label{sec:poiseuille}
Under the lubrication approximation, the conservation of mass for the fluid inside
the fracture leads to the following equation:
$$  \pa_t u -\pa_x\left(\frac{u^3}{12\mu}\pa_x p\right)=0$$
where $p(x)$ denotes the pressure exerted on the fluid by the 
rock and $\mu$ is the viscosity coefficient of the fluid (see \cite{im} for more details about the lubrication approximation).

\paragraph{The pressure law.}
In the very simple geometry that we consider here, the elasticity
equation expresses the pressure as a function of the fracture opening.
More precisely, after a rather involved computation, which is recalled
in Appendix \ref{app:pressure} \cite{CS83,Peirce}, we obtain:
\begin{equation}\label{eq:pressure0}
 p(x) = \frac{E}{4(1-\nu^2)} (-\Delta)^{1/2} u
 \end{equation}
 where the square root of the Laplacian $(-\Delta)^{1/2}$ is defined
 using Fourier transform by \[\mathcal F((-\Delta)^{1/2} u)(k)=|k|
 \mathcal F(u)(k),\] and $E$ denotes Young's modulus and $\nu$ is
 Poisson's ratio. We use the following convention for the Fourier transform,
\[ \mathcal{F} f (\xi) = \int_{\R} f(x) e^{-ix\xi} dx . \]

\paragraph{Propagation condition (Free boundary condition)}
Equation (\ref{eq:tf0}) is satisfied only inside the fracture, that is
in the support of $u$.  It must be supplemented with boundary
condition on $\pa\{u>0\}$ (the free boundary).  Naturally, we impose
$$ 
u=0,\qquad u^3 \pa_x p = 0 \qquad \mbox{ on } \pa\{u>0\} 
$$
which ensures zero width and zero fluid loss at the tip of the
fracture.  However, because we have an equation of order three, and
the support is not known a priori, we need an additional condition to
fully determine the solution.  This additional condition is a
propagation condition which requires the rock toughness $K_{IC}$
(which is given) to be equal to the stress intensity factor $K_I$ at
the tip of the fracture.  If $\{u>0\}=(a(t),b(t))$, then the stress
intensity factor at $x=b(t)$ is defined by
 $$ K_I : = \lim_{x\to \pa b^+} \sqrt{2\pi} \sqrt{x-b}\, \sigma_{yy}(x,0)$$
 where $\sigma_{yy}$ is the $yy$ component of the stress tensor given by
  (see Appendix \ref{app:pressure}):
 $$ \sigma_{yy}(x,0) = -p(x).$$
 So the propagation condition prescribes the behavior of the pressure
 at the tip of the fracture (outside of the fracture).  A simple but
 technical lemma (see Appendix~\ref{app:bb} for a proof) shows that
 this is related to the behavior of $u$ inside the fracture:
\begin{lemma}\label{lem:derfb}
  Assume that $\{u>0\}=(-1,1)$ and recall that $p$ is defined by
  \eqref{eq:pressure0}. Then we have the following relations
\begin{equation}\label{eq:asymp1}
\lim_{x\to 1^+}   - \sqrt{x-1}\, p(x) = \frac{1}{\pi \sqrt 2}\int_{-1}^1\frac{\sqrt{1+z}}{\sqrt{1-z}} p(z)\, dz
\end{equation}
and
\begin{equation}\label{eq:asymp2}
\lim_{x\to 1^-} u'(x) \sqrt{1-x} = -\frac{4(1-\nu^2)}{\pi \sqrt 2 E}\int_{-1}^1 \frac{\sqrt{1+z}}{\sqrt{1-z}} p(z)\, dz
\end{equation}
\end{lemma}
In view of this lemma, the propagation condition $K_I=K_{IC}$ is thus
equivalent to (assuming $b=1$):
$$
u(t,x) \sim\sqrt{\frac 2 \pi} \frac{4(1-\nu^2)}{E} K_{IC} \sqrt{1-x}
\qquad \mbox{ as } \quad x\rightarrow 1^-
$$
which is the free boundary condition \eqref{eq:bc2}.

In the literature (see for instance \cite{MKP2,MKP}), this relation is
often written as
$$
u(t,x) \sim\frac {K'}{E'}  \sqrt{1-x}  \qquad \mbox{ as } \quad x\rightarrow 1^-
$$
where $K'=4 \sqrt{\frac 2 \pi}  K_{IC}$ and $E'=\frac{E}{1-\nu^2}$.

\section{Preliminary}

\subsection{The zero toughness case (Theorem \ref{thm:main}-(i))}
\label{sec:notough}

When $h(t)=0$ (no injection of fluid), equation \eqref{eq:ftf}
preserves the total mass, and so in order to find a self-similar
solution of the form \eqref{eq:profile} we must take $\alpha=\beta$.
According to \eqref{eq:tout}, the free boundary condition
\eqref{eq:tough} can then only be satisfied for all time if we have
$K=0$ (there also exist solutions with $K\neq 0$ depending on $t$, but
the physical meaning of such solutions is not clear).

Next, we note that the condition~\eqref{eq:ab}, with $\alpha=\beta$, implies
\[\alpha=\beta= \frac1{n+3},\]
and equation \eqref{eq:UU} becomes
\[-(xU)' +(n+3)(U^n I(U)')'  =0 \mbox{ in } (-a,a).\]
We can integrate this equation once, and using the null flux boundary
condition \eqref{eq:bcU}, we find
\begin{equation}\label{eq:UUMM}
(n+3) U^n (I(U))' = x U  \qquad \text{ in } [-a,a].
\end{equation}
At the end points $\pm a$, we have the obvious condition $U(\pm a)=0$,
and condition~\eqref{eq:tough} (with $K=0$) can also be written as
\[ U(x) = o \left(\sqrt{|a^2-x^2|} \right)\text{ as } x\to \pm a.\]

We recall that we also have the mass condition $\int_{-a}^a U(x)\,
dx=m$. However, instead of fixing the mass, we will fix $a=1$ and
ignore the mass condition.  Indeed, if $U$ solves \eqref{eq:UUMM} in
$(-1,1)$, then  $V(x)=a^{3/n}U(x/a)$ solves \eqref{eq:UUMM} in $(-a,a)$ and
satisfies $\int_{-a}^a V(x)\, dx=m$ provided we choose
$a^\frac{3+n}{n}=m/\int_{-1}^1 U(x)dx$.

We can also remove the multiplicative factor $n+3$ (consider the
function $V(x) = b U(x)$ with $b= (n+3)^\frac 1 n$).
\medskip

In conclusion, our task will be to prove that there exists a profile
function $x\mapsto U(x)$ solution of
\begin{equation}\label{eq:self}
\begin{cases} 
U^n I(U)' = x U & \text{ for } x \in (-1,1) \\
U = 0 & \text{ for } x \notin (-1,1) \\
U = o((1-x^2)^{\frac12}) & \text{ for } x \sim \pm 1.
\end{cases}
\end{equation}

\begin{remark}\label{rem:casen1}
Note that for $n=1$, the equation reduces to $I(U)'=x$, which has an
explicit solution (see \cite{bik}):
\[ U(x) = \frac49(1-x^2)_+^{\frac32}.\]
See Lemma~\ref{lem:casen1} in Appendix for a proof of this fact.
\end{remark}
\medskip

So the first part of Theorem \ref{thm:main} is a consequence of the
following proposition:
\begin{proposition}\label{prop:1}
For all $n\in [1,4)$, there exists a non-negative even function $U\in
  C^1(\R)\cap C^\infty_{loc}(-1,1)$ such that $U>0$ in $(-1,1)$ and
  solving \eqref{eq:self}.
\item Furthermore, $U$ satisfies
\begin{equation} \label{eq:Uequiv}
 U(x)= \left\{
\begin{array}{ll}
C^*(1-x^2)^{\frac{3}{2}} +  \mathcal O (|1-x^2|^{\frac 2  n}) & \mbox{ if } n \in [1,\frac{4}{3})\\
C^*(1-x^2)^{\frac32} |\ln(1-x^2)|^{\frac34} +  \mathcal O (|1-x^2|^{\frac 32})  & \mbox{ if } n= \frac{4}{3}\\
C^*(1-x^2)^{\frac{2}{n}} +  o (|1-x^2|^{\frac 2  n}) & \mbox{ if } n \in  (\frac{4}{3},4) 
\end{array}
\right.\end{equation}
when $x\to\pm 1$ for some positive constant $C^*>0$. 
\end{proposition}

\subsection{The finite toughness case (Theorem \ref{thm:main}-(ii))} 
\label{sec:tough}

When the toughness coefficient $K$ is not zero, then \eqref{eq:tout} imposes
\[ \alpha =  - \frac{\beta}{2},\]
  and using \eqref{eq:ab} we see that we must have $n \neq 6$ and
\[ \alpha = - \frac{\beta}{2}= - \frac{1}{6-n}.\]
  
In particular, in view of \eqref{eq:h} we see that a self-similar
solution can only exists in that case if the injection rate has the
form
\[ h(t)=\lambda t^{\frac{n-3}{6-n}}.\]
Equation \eqref{eq:UU} can then be written as
\begin{equation}\label{eq:Ulambda} 
(-\beta x U + U^n I(U)')' =  \lambda \delta -  \frac32 \beta U.
\end{equation}
We now choose $a>0$ and try to solve \eqref{eq:Ulambda}  on the interval $(-a,a)$.
If we integrate this equation on $(-a,a)$, we see that the null-flux boundary condition \eqref{eq:bcU} implies a compatibility condition between $\lambda$ and the mass of $U$:
\[ \lambda = \frac32 \beta m
\qquad \text{ with } \quad m= \int_{-a}^a U (x) dx.\]
We  can now eliminate $\lambda$ from \eqref{eq:Ulambda}:
The profile $U(x)$ must solve the following equation:
$$
(-\beta x U + U^n I(U)')' = \frac32 \beta (m \delta -  U) \quad \mbox{ with } m = \int_{-a}^a U(x) \, dx .
$$
Integrating and using  \eqref{eq:bcU}, we thus find
\begin{equation}\label{eq:Unolambda} 
U^n I (U)' = \beta x U + \frac32 \beta \mathcal{U}  \mbox{ in } (-a,a)
 \end{equation}
where
\[\mathcal{U} (x)=
\begin{cases}
\int_x^a U (y) \, dy  & \text{ if } x > 0, \\
-\int_{-a}^x U (y) \, dy & \text{ if } x <0.
\end{cases}\]

\medskip

We thus need to construct a solution of \eqref{eq:Unolambda} satisfying
 \begin{equation}\label{eq:UKK} 
 U(x) = K |x-a|^{1/2} + o(|x-a|^{1/2}),
 \end{equation}
 for a given $K>0$.
Any such solution will  solve \eqref{eq:Ulambda} for the particular choice of $\lambda$ given by
\begin{equation}\label{eq:lambdaa}
\lambda(a) = \frac32 \beta \int_{-a}^a U (x) dx.
\end{equation}
\medskip

As before, we see that we can always take $a=1$ and get rid of the
parameter $\beta$ in the equation by considering the function $V(x) =
b U (ax)$ with $b$ such that
\[\beta b^n a^3 =1.\]
Note that condition~\eqref{eq:UKK} can then be written as 
$$ 
V(x) =
K'\sqrt{1-x^2} + o(\sqrt{1-x^2})$$ 
with $K' = \frac{Kb\sqrt a}{\sqrt 2}$.  
  \medskip

In Section \ref{sec:finite} (see Proposition \ref{prop:2} below), we will prove the existence of such a $V(x)$. 
This implies that for any $K>0$ and $a>0$,  Equation \eqref{eq:Ulambda} has a solution
for a particular value of $\lambda$ (given by \eqref{eq:lambdaa}).
As noted in Remark \ref{rem:injec-n3}, we would like to say that for  given $K>0$ and $\lambda_0$, we can always find $a>0$ such that $\lambda(a)=\lambda_0$. 
While we are unable to prove that fact, we do want to point out that
Lemma \ref{lem:unif} will give 
\[ V (x) \ge K' (1-x^2)^{\frac12} \qquad \mbox{ for all } x\in(-1,1) \]
and 
\[ V (x) \leq C({K'}^{1-n}+K') (1-x^2)^{\frac12}\]
for a constant $C$ depending only on $n$. We deduce
\[C^{-1} K'  \leq \int_{-1}^1 V(x)\, dx  \leq C ({K'}^{1-n}+K'),\]
and the corresponding function $U$ will thus satisfies
\[C^{-1} K a^{\frac 32} \leq \int_{-a}^a U(x)\, dx \leq C(a^{\frac{9-n}{2}} K^{1-n} + a^{3/2} K).\]
Using \eqref{eq:lambdaa}, we deduce that for $K>0$ fixed we have 
\[\lim_{a\to 0}  \lambda (a)=  0 \quad \mbox{ and } \quad \lim_{a\to\infty} \lambda (a)= \infty. \]
  It is thus reasonable to expect that $\lambda(a)=\lambda_0$ for some $a$ (but, as noted in Remark \ref{rem:injec-n3}, one needs to establish the continuity of $a\mapsto \lambda(a)$ in order to conclude).

\medskip

In conclusion, it is enough to solve \eqref{eq:Unolambda} when $a=1$
and $\beta=1$. So we have to construct, for any $K>0$, a solution
$U(x)$ of
\begin{equation}\label{eq:self-bis}
\begin{cases} 
 U^n I (U)' = x U + \frac32\mathcal{U} & \text{ for } x \in (-1,1) \\
U = 0 & \text{ for } x \notin (-1,1) \\
U = K \sqrt{1-x^2} +o ((1-x^2)^{\frac12}) & \text{ when } x \to \pm 1
\end{cases}
\end{equation}
where
\[\mathcal{U} (x)=
\begin{cases}
\int_x^1 U (y) \, dy  & \text{ if } x > 0, \\
-\int_{-1}^x U (y) \, dy & \text{ if } x <0.
\end{cases}\]
The second part of Theorem \ref{thm:main} is thus an immediate
consequence of the following proposition:
\begin{proposition}\label{prop:2}
For all $n\in [1,4)$, there exists a non-negative even function $U\in
  C^{1/2}(\R)\cap C^\infty_{loc}(-1,1)$ such that $U>0$ in $(-1,1)$
  and $U$ solves \eqref{eq:self-bis}.  Furthermore, $U$ satisfies
\begin{equation}\label{eq:Uasympttough} 
U(x)= K \sqrt{1-x^2} + 
\left\{ 
\begin{array}{ll}
\mathcal O((1-x^2)^{\frac{3}{2}}) & \mbox{ if } n \in  [1,2) \\
\mathcal O((1-x^2)^{\frac32} |\ln(1-x^2)|) & \mbox{ if } n=2\\
\mathcal O((1-x^2)^{\frac{5-n}{2}}) & \mbox{ if } n \in (2,4)
\end{array}
\right.
\end{equation}
when $x\to \pm1$.
\end{proposition}

\subsection{General strategy}\label{sec:strategy}

In order to show the existence of even solutions to \eqref{eq:self}
and \eqref{eq:self-bis}, we will follow the general approach used in
\cite{bf} to prove the existence of source-type solutions for the thin film equation. 
The first step is to rewrite these equations as integral
equations by introducing an appropriate Green function. More precisely,
we consider the function $x\mapsto g(x,z)$ solution of (for all $z \in
[-1,1]$)
\begin{equation} \label{eq:green}
\left\{
\begin{array}{ll}
I(g(\cdot,z))' = \frac{1}{2} \Big[ \delta_z -\delta_{-z} \Big]  &
\text { for } x \in (-1,1)\\
g(x,z)=0 & \text { for } x \notin (-1,1)\\
g(x,z) = \mathcal{O} ((1-x^2)^{\frac32}) & \text{ for } x \sim \pm 1.
\end{array}
\right.
\end{equation}

In particular, formally at least, for any even function $V(x)$
satisfying $V(x)=0$ for all $x \notin (-1,1)$, the function
\[ U(x)= \int_{-1}^{1} g(x,z) z V(z) \, dz\]
solves
\[ I(U)' = z V \qquad \mbox{ in } (-1,1).\]

We can thus rewrite equation \eqref{eq:self} as
\[ U (x) = \int_{-1}^1 g(x,z) z U^{1-n} (z) dz , 
\quad x \in [-1,1] \]
and equation \eqref{eq:self-bis}  as
\[ U (x) = 2\int_0^1 g(x,z) U^{-n} \left( z U(z)+ \frac32
  \mathcal{U}(z) \right) dz + K \sqrt{1-x^2}, \quad x \in [-1,1]. \]
Solutions of these integral equations will be obtained via a fixed
point argument in an appropriate functional space. One of the main
difficulty in developing this fixed point argument is the fact that
for $n>1$ (see Remark~\ref{rem:casen1}), the function $U^{1-n}$ is singular at the endpoints $\pm
1$.  Another difficulty will be to show that the solution has the
appropriate behavior at $\pm 1$.  These two difficulties are in fact
clearly related, and both will require us to have a very precise
knowledge of the behavior of the Green function $g$ as $x$ and $z$
approach $\pm 1$.  This will be the goal of the next section.

\section{Properties of the  Green function}\label{sec:Green}

In this section, we are going to derive the
formula for the Green function $g(x,z)$ solution of \eqref{eq:green}
and study its properties (in particular its behavior near the
endpoints $\pm 1$).

\subsection{Green function for $(-\Delta)^{1/2}$}

First, we recall that the Green function for the square root of the
Laplacian in $[-1,1]$ with homogeneous Dirichlet conditions, that is
the solution of 
\[\begin{cases}
-I(G) = \delta_y & \text{ in } (-1,1) \\
G = 0 & \text{ in } \RR\setminus (-1,1)
\end{cases}\]
is given in \cite{riesz,bgr} by the formula:
\begin{equation}\label{eq:G1}
G(x,y) := \begin{cases}\pi^{-1} \argsinh(\sqrt{r_0(x,y)}) & \text{
    if } x,y \in (-1,1),\\
0 & \text{ otherwise} \end{cases}
\end{equation}
with 
\[r_0(x,y) = \frac{(1-x^2)(1-y^2)}{(x-y)^2}.\]

Equivalently, we have the following formula for $x,y \in (-1,1)$, 
\begin{equation}\label{eq:G2} 
G(x,y) = \pi^{-1}\ln \left( \frac{1-xy +
  \sqrt{(1-x^2)(1-y^2)}}{|x-y|}\right).
\end{equation}
(Eq.~\eqref{eq:G2} follows from \eqref{eq:G1} using the relation $\argsinh(u) =
\ln(u+\sqrt{u^2+1}))$).

We give the following lemma for the reader's sake (see also \cite{riesz} and \cite[Corollary~4]{bgr}):
\begin{lemma}[Green function of $(-\Delta)^{\frac12}$]\label{lem:greenG}
The function $G$ defined above  satisfies, for all $y \in (-1,1)$,
\[ - I (G(\cdot,y) ) = \delta(\cdot - y)  \quad \text{ in }  \mathcal{D}'((-1,1)). \]
In particular, for any  function $f:(-1,1)\to\R$ satisfying
\begin{equation}\label{eq:fb}
 f(x)\leq C (1-x^2)^b
 \end{equation}
for some $b>-\frac{3}{2}$, 
the function defined by 
\begin{equation}\label{eq:intu} 
u(x)=\int_{-1}^1 G(x,y) f(y)\,dy
\end{equation}
is continuous in $(-1,1)$ and it satisfies
\[ -I(u) = f \qquad \mbox{ in }  \mathcal D'(-1,1).\]
\end{lemma}
\begin{proof}
Computations were first made in \cite{riesz}. The validity of formulas
in the one-dimensional setting were established in
\cite{bgr}. 
So we just want to prove that the integral \eqref{eq:intu} is finite
for all $x\in (-1,1)$ under condition \eqref{eq:fb}. The rest of the proof follows as in \cite{bgr}.

For that purpose, we fix $x\in[0,1)$ (the case $x\in(-1,0]$ would be treated similarly) and denote $\eps=\frac{1-x}{2}$.
We then write:
\[ |u(x)|\leq \left| \int_{-1+\eps}^{1-\eps} G(x,y) f(y)\, dy \right|
+\left|  \int_{1-\eps}^1 G(x,y) f(y)\, dy \right| +  \left| \int_{-1}^{-1+\eps} G(x,y) f(y)\, dy\right|\]
To bound the first term, we use Formula \eqref{eq:G2} to get
\begin{align*}
\left|  \int_{-1+\eps}^{1-\eps} G(x,y) f(y)\, dy\right|  
&\leq  C(1+\eps^b)   \int_{-1+\eps}^{1-\eps} |G(x,y)| \, dy \\
& \leq C (1+\eps^b) \int_{-1+\eps}^{1-\eps} ( |\ln\eps|+|\ln|x-y||) \, dy\\
 &\leq C (1+\eps^{b})(|\ln\eps|+1)
 \end{align*}
where we used the fact that
$$\eps\leq 1-xy \leq 1-xy +  \sqrt{(1-x^2)(1-y^2)} \leq 3  \qquad\forall |y|\leq 1-\eps, \mbox{ with } x=1-2\eps .$$
In order to bound the last two terms, we use formula \eqref{eq:G1} and
the fact that $\argsinh(u)\leq u$ for all $u\geq 0$ to get
\begin{align*} 
 \int_{1-\eps}^1 G(x,y) f(y)\, dy  
 & \leq C   \int_{1-\eps}^1 \sqrt{r_0(x,y)}f(y)\, dy  \\
&  \leq C  \frac{\sqrt{1-x^2}}{\eps}\int_{1-\eps}^1 \sqrt{(1-y^2)}f(y)\, dy \\
& \leq C \eps^{b+1}.
\end{align*}
We have thus showed that 
$$|u(x)|\leq h(1-x)<\infty \mbox{ for all }x\in(-1,1)$$
for some function $h$ which satisfies in particular $h(y)\leq (1+y^{b})(|\ln y|)$
(this inequality is far from optimal, as we will see later on).
\end{proof}

\subsection{Green function for Equation~\eqref{eq:green}}

We now claim that the Green function $g(x,z)$, solution of
\eqref{eq:green}, is given by
\begin{equation}\label{eq:gzz}
g(x,z) = \frac{1}{2}\left[ (z-x) G(x,z)+(z+x)G(x,-z)\right] .
\end{equation}
More precisely, we have the following proposition.
\begin{proposition}[A Green function for a higher order operator]\label{prop:def-g}
For all $z \in (-1,1)$, the function $x\mapsto g(x,z)$ defined by
\eqref{eq:gzz} is the unique solution of
\begin{equation}\label{eq:ggg}
\left\{
\begin{array}{ll}
I(g(\cdot,z))' = \frac{1}{2} \Big[ \delta_z -\delta_{-z} \Big]  &
\text { in } \mathcal D' (-1,1),\\
g(x,z)=0 & \text { for } x \in \RR\setminus (-1,1),\\
g(x,z) = o ((1-x^2)^{\frac12}) & \text{ when } x \to \pm 1.
\end{array}
\right.
\end{equation}
\end{proposition}

Before proving this result, we give two simple but useful lemmas.
\begin{lemma}\label{lem:derivatives-G}
The partial derivatives of $G$ are given by the following formulas
\begin{align*}
 \frac{\partial G}{\partial x} (x,y) = \pi^{-1}
\frac{\sqrt{1-y^2}}{\sqrt{1-x^2}} \frac1{y-x}, \\
\frac{\partial G}{\partial y} (x,y) = \pi^{-1}
\frac{\sqrt{1-x^2}}{\sqrt{1-y^2}} \frac1{x-y}.
\end{align*}
\end{lemma}
\begin{proof}[Proof of Lemma \ref{lem:derivatives-G}] 
Remark that $G(x,y) = G(y,x)$; hence, it is enough to prove one of the
two formulas. To prove the first one,  simply write 
\[\pi \frac{\partial G}{\partial x} = \frac{\partial_x (1-xy +
  \sqrt{1-x^2} \sqrt{1-y^2})}{1-xy + \sqrt{1-x^2}\sqrt{1-y^2}} -
\frac{1}{x-y} .\] 
A rather long but straightforward computation gives the desired
result. 
\end{proof}
Furthermore, with a simple integration by parts using Lemma
\ref{lem:derivatives-G} (see Appendix for details) we get the
following lemma.
\begin{lemma} \label{lem:asymp2}
For all $z\in(-1,1)$ and $x\in(-1,1)$, we have
\begin{equation}\label{eq:GG1}
\int_{-z}^z G(x,y)\, dy   =  (z-x) G(x,z)+(z+x)G(x,-z) 
+\frac 2\pi \sqrt{1-x^2}\arcsin(z).
\end{equation}
\end{lemma}

We now turn to the proof of Proposition~\ref{prop:def-g}.
\begin{proof}[Proof of Proposition~\ref{prop:def-g}]
We will actually derive formula \eqref{eq:gzz}:
Integrating the equation
\[I(g(\cdot,z))' = \frac{1}{2} \Big[ \delta_z -\delta_{-z} \Big]\]
with respect to $x$, we find that the function $x\mapsto g(x,z)$ must solve
\begin{equation}\label{eq:dfdfi}
I(g(\cdot,z)) = \frac{1}{2} \Big[ H(x-z) -H(x+z)\Big] +a(z)
\end{equation}
for some $a(z)$, where $H$ is the Heaviside function satisfying
$ H(x)=1$ for $x\geq 0$, $H(x)=0$ otherwise.

Together with the boundary condition $g(x,z)=0$ for $x\notin(-1,1)$,
\eqref{eq:dfdfi} has a unique solution given by Lemma \ref{lem:greenG}:
\begin{eqnarray} 
g(x,z) & = & -\int_{-1}^1 G(x,y)\left[ \frac{1}{2} \Big[ H(y-z)
    -H(y+z)\Big] +a(z)\right] \,dy \nonumber \\
&  = &  \frac{1}{2} \int_{-z}^z G(x,y)\, dy - a(z)\int_{-1}^1 G(x,y)\,
dy
\label{eq:gG}. 
\end{eqnarray}

Note that since $G(x,1)=G(x,-1)=0$, \eqref{eq:GG1} with $z\rightarrow
1$ gives
\[\int_{-1}^1 G(x,y)\, dy  =\frac2\pi\sqrt{1-x^2}\arcsin(1) 
=  \sqrt{1-x^2}\]
and \eqref{eq:gG}  thus gives
\begin{equation}\label{eq:gzz-bis}
g(x,z) = \frac{1}{2}\left[ (z-x) G(x,z)+(z+x)G(x,-z)\right] + b(z)
\sqrt{1-x^2} 
\end{equation}
with $b(z) = -\frac1{\pi}[a(z) \pi-\arcsin(z)] $. 
\medskip

Finally, the function $a(z)$ (and thus $b(z)$) will be determined
uniquely using the last boundary condition in \eqref{eq:ggg}.  Indeed,
using the fact that $\argsinh (u) \sim u$ as $u\to 0$, we get
\[G(x,z)\sim \pi^{-1}\sqrt{r_0(x,z)}\]
when either $x\rightarrow \pm 1$ with $z$ fixed, or when $z\rightarrow
\pm 1$ with $x$ fixed.  We deduce
\begin{align*}
g(x,z) & \sim
\frac{1}{2\pi}\left[ (z-x)\sqrt{r_0(x,z)} + (z+x)\sqrt{r_0(x,-z)} \right]
+ b(z) \sqrt{1-x^2}
\\ & \sim \frac{1}{2\pi}\left[ \frac{(z-x)}{|z-x|} + \frac{(z+x)}{|z+x|}
  \right] \sqrt{1-x^2}\sqrt{1-z^2} + b(z) \sqrt{1-x^2}.
\end{align*}
Hence, $g$ satisfies
\[g(x,z) = o ((1-x^2)^{\frac12}) \quad \text{ when } x \to \pm 1\]
for all $z\in(-1,1)$
if and only if we choose $b(z)=0$ (that is $a(z)=\frac{1}{\pi}\arcsin(z)$).
The proof of the proposition is now complete.
\end{proof}

\subsection{Further properties of $g(x,z)$}

  The following proposition summarizes the properties of $g$ that 
  will be needed for the  proof of our main result.
\begin{proposition}[Properties of the function $g$]\label{prop:g}
\begin{enumerate}[\upshape 1).]
We have:
\item  \label{inc}
The function $g$ is continuous on $(-1,1)^2$ and
for all $x,z \in (-1,1)$ with $x\neq z$ and $x\neq -z$, we have
\begin{eqnarray}
\frac{\partial g}{\partial x} (x,z) 
& =& - \frac{1}{2} \left[  G(x,z)-G(x,-z)\right] \label{eq:xderG}\\
& =& - \frac1{2\pi} \argsinh\left( \frac{2 xz
  \sqrt{(1-x^2)(1-z^2)}}{|x-z||x+z|} \right).\label{eq:xder}
\end{eqnarray} 
In particular,  $x \mapsto g(x,z)$    is decreasing on $(0,1)$ for all $z\in(0,1)$. 

\item \label{pos} 
We have 
\[ g(x,-z)=-g(x,z) \mbox{ and } g(-x,z)=g(x,z) \quad\mbox{ for all $(x,z)\in \R^2$}\]
so the function $z\mapsto g(x,z)$ is odd and the function $x\mapsto g(x,z)$ is even.
Furthermore, $g$ satisfies
\[g (x,z) > 0\quad  \mbox{ for all $x \in (-1,1)$ and for all $z\in(0,1)$}\]
(and so $g(x,z)<0$ for all $x \in (-1,1)$ and for all $z\in(-1,0)$).

\item \label{xder-estim} For all $x,z \in (-1,1)$,
\begin{eqnarray} \label{estim:xzG}
|g(x,z)|  & \leq &    \frac{1}{\pi} \sqrt{1-x^2}\sqrt{1-z^2},  \\
\label{eq:xder-estim}
\left|\frac{\partial g}{\partial x} (x,z) \right| &\le & \frac{1}{2\pi} \ln \left(1 +
\frac{4\sqrt{1-x^2}\sqrt{1-z^2}}{|z^2-x^2|}\right). 
\end{eqnarray}
\item  For all $x \in (-1,1)$, 
\begin{equation}\label{non-deg} 
 \int_x^1 z  g(x,z) dz \ge C (1-x^2)^2 
\end{equation}
for some $C>0$.
\end{enumerate}
\end{proposition}
\begin{proof}[Proof of Proposition \ref{prop:g}]
  
\item[1.] The continuity of $g$ is easy to check. Indeed, the only
  singularity for the function $G(x,z)$ occurs when $x=z$, and since
  it is a logarithmic singularity, it is clear that the function
  $(z-x) G(x,z)$ is continuous everywhere.

Next, we have
\begin{align*} \frac{\pa g}{\pa x}    (x,z)
 = &  \frac{1}{2} \left[ -G(x,z)+G(x,-z)\right] \\
   & + \frac{1}{2} \left[ (x+z)\frac{\pa G}{\pa x} (x,-z)-(x-z)\frac{\pa
  G }{\pa x} (x,z)\right] 
\end{align*}
and using Lemma~\ref{lem:derivatives-G}, we find
\[(x-z)\frac{\pa G}{\pa x} (x,z)=   
 (x+z)\frac{\pa G}{\pa x} (x,-z) = -\frac{\sqrt{1-z^2}}{\sqrt{1-x^2}}.\]
We deduce
\[ \frac{\pa g}{\pa x}    (x,z) = - \frac{1}{2} \left[
  G(x,z)-G(x,-z)\right]. \]
  The last formula follows from the identity
\[\argsinh(u)-\argsinh(v) = \argsinh\Big (u \sqrt{1+v^2} - v
\sqrt{1+u^2}\Big).\]

\item[2.]  The fact that $z\mapsto g(x,z)$ is odd and $x\mapsto
  g(x,z)$ is even is a direct consequence of the
  formulas~\eqref{eq:gzz} and \eqref{eq:G1}.  The positivity of $g$
  follows from the monotonicity and the fact that $g(1,z)=0$ for all
  $z$ (see \eqref{estim:xzG} for instance).

\item[3.] Since $\argsinh (u) \leq u$ for all  $u\geq  0$, we have 
\[ G(x,z)\leq \sqrt{r_0(x,z)}\] 
and so
\begin{align*}
|g(x,z)| & \leq \frac{1}{2\pi}\left[  |z-x|\sqrt{r_0(x,z)} +  |z+x|\sqrt{r_0(x,-z)} \right] \\
& \leq  \frac{1}{\pi} \sqrt{1-x^2}\sqrt{1-z^2} .
\end{align*}

To prove \eqref{eq:xder-estim}, we use the fact that for $u\geq 0$, we
have $\sqrt{1+u^2} \leq 1+u$, and so
\[ \argsinh (u) =\ln(u+\sqrt{1+u^2}) \le \ln(1+2u).\]
Inequality \eqref{eq:xder-estim} now follows from \eqref{eq:xder} for
$x,\, z\in (0,1)$.  The symmetries of $g$ then give the result for
$x,\, z\in (-1,1)$.

\item [4.] In order to prove \eqref{non-deg}, we write for $x \in [0,1)$, 
\begin{align*} 
\int_x^1 z g(x,z) dz =& \int_x^1 [ z (x+z) G(x,-z) + z(z-x) G(x,z) ] dz \\
& \ge  \int_x^1  z (z-x) G(x,z) \, dz.
\end{align*}
Integrating by parts, we get 
\begin{align*} 
\int_x^1 z(z-x)  G(x,z) dz = &  \frac13 \int_x^1 \partial_z G(x,z) (z-x)^2
(z+\frac x 2) dz \\
& + \frac 1 3 \left(G(x,z)(z-x)^2(z+\frac x 2)\right)\Big|_{z=x}^{z=1}. 
\end{align*}
Keeping in mind that $(z-x)G(x,z)$ is not singular and vanishes when
$z=x$ and using the formulas for partial derivatives of $G$ (Lemma
\ref{lem:derivatives-G}), we obtain
\[\int_x^1 z(z-x) G(x,z) dz = \frac{\sqrt{1-x^2}}{3\pi} \int_x^1
(z-x)(z+\frac x 2) \frac{dz}{\sqrt{1-z^2}} \]
where
\begin{align*} 
 \int_x^1 (z-x)(z+\frac x 2)  \frac{dz}{\sqrt{1-z^2}}  
&  =\frac{1}{2}(1-x^2) (\pi/2-\arcsin(x))\\
 &  =\frac{1}{2}(1-x^2)\arccos(x).
 \end{align*}
Since $\arccos (x) \geq \sqrt{1-x^2}$ the result follows.
 \end{proof}

\subsection{Application: Solving the  linear problem}

In this subsection, we use the Green function $g(x,z)$ introduced above
to find a solution to the linear equation
\begin{equation}\label{eq:linearIU} 
I(U)' = f \quad \mbox{ in } (-1,1), \quad U=0 \quad \mbox{ in  } \R\setminus (-1,1)
\end{equation}
and to study the behavior of this solution $U$ as $x\to\pm 1$.

We note that the function $V(x) = \sqrt{(1-x^2)_+}$ solves $I(V)'=0$
in $(-1,1)$, and so given one solution $U_0$ of \eqref{eq:linearIU},
we can find all solutions in the form $U_0+KV(x)$ (and there is a
unique solution to \eqref{eq:linearIU} if we add a boundary condition
such as \eqref{eq:tough}).
\medskip

Now,  we start with the following result.
\begin{proposition}\label{prop:linear-g}
Let $f\colon (-1,1) \to \R$ be a function satisfying
\begin{equation}\label{hyp-resol} 
|f(z)| \le C_f (1-z^2)^a
\end{equation}
for some $a > -\frac32$. 
Then the function
\begin{equation}\label{def:U}
 U(x) = \begin{cases} \int_{-1}^1 g(x,z) f(z) dz & \text{ for } x
  \in (-1,1) \\ 0 & \text{ for } x \notin (-1,1) 
\end{cases}
\end{equation}
is  continuous in $\R$,  $C^1$ in $(-1,1)$ and satisfies
\begin{equation}\label{eq:bdU} 
|U(x)| \leq C C_f  \sqrt{1-x^2} \qquad \quad\forall x\in(-1,1)
\end{equation}
for some constant $C$ depending on $a$.
\item Furthermore, if $f$ is odd, then $U$
solves
\begin{equation}\label{lin}
I(U)'=I(U') = f \qquad \mbox{ in } \mathcal D' (-1,1).
\end{equation}
\end{proposition}
\begin{proof}[Proof of Proposition~\ref{prop:linear-g}]
First of all, \eqref{estim:xzG} implies (for $x\in (-1,1)$):
\begin{align*}
 |U(x) | & \leq \sqrt{1-x^2}\int_{-1}^1\sqrt{1-z^2} | f(z)| dz \\
 & \leq C_f  \sqrt{1-x^2}\int_{-1}^1(1-z^2)^{\frac{1}{2}+a}dz
\end{align*}
where this last integral is clearly convergent for $a>-\frac{3}{2}$. We deduce \eqref{eq:bdU}
which gives in particular the continuity of $U$ at $\pm1$ (the continuity of $U$ in $(-1,1)$ is clear).

Furthermore, we have (using \eqref{eq:xderG}):
\begin{align}
 U'(x) & = \int_{-1}^1  \frac{\partial g}{\partial x}(x,z) f (z) dz\nonumber \\
  & =-\frac{1}{2}\int_{-1}^1 [G(x,z)-G(x,-z)] f(z)\, dz\label{eq:U'}
\end{align}
and Lemma \ref{lem:greenG} implies that $U'$ is continuous  in $(-1,1)$ and satisfies
\[ I(U') (x)= \frac{1}{2}[f(x)-f(-x)] \qquad \mbox{ in } \mathcal D'(-1,1).\]
In particular, if $f$ odd, we deduce
\[ I(U')=f  \qquad \mbox{ in } \mathcal D'(-1,1).\]
Finally, Proposition \ref{prop:def-g} also  implies that
\[I(U)' (x)= \frac{1}{2}[f(x)-f(-x)] =f(x) \qquad \mbox{ in } \mathcal
D'(-1,1).\qedhere\]
\end{proof}

In the proof of our main result, we will need to further characterize
the behavior of the function $U$ near the end points $x=\pm 1$.  We
thus prove the following proposition.
\begin{proposition}\label{prop:estim-derU}
Consider an odd function $f\colon (-1,1) \to \R$
satisfying \eqref{hyp-resol} 
for some $a > -\frac32$.

Then the function $U$ defined by \eqref{def:U} satisfies
\begin{equation}\label{eq:estimU'} 
|U'(x)| \le C C_f F(1-x^2)
\end{equation}
with
\begin{equation}\label{eq:F}
F(y) = \left\{ 
\begin{array}{ll}
y^{a+1} & \mbox{ if } -\frac32 < a< - \frac{1}{2} \\
y^\frac 1 2 \ln\left(\frac{1}{y}\right) & \mbox{ if } a= - \frac{1}{2}\\
y^\frac 1 2 & \mbox{ if } a> - \frac{1}{2}.
\end{array}
\right.
\end{equation}
\end{proposition}

Together with the fact that $U(\pm 1 ) = 0 $ (which follows from \eqref{eq:bdU}), this proposition gives
\begin{equation}\label{eq:ffrU}
U(x)\leq C 
\left\{ 
\begin{array}{ll}
(1-x^2)^{a+2} & \mbox{ if } -\frac32 < a< - \frac{1}{2} \\
(1-x^2)^\frac 3 2 \ln\left(\frac{1}{1-x^2}\right) & \mbox{ if } a= - \frac{1}{2}\\
(1-x^2)^\frac 3 2 & \mbox{ if } a> - \frac{1}{2}.
\end{array}
\right.
\end{equation}
In particular, we have 
$$ U(x) = o(\sqrt{1-x^2}) \mbox{ as } x\to\pm1.$$

\begin{remark}
We will apply estimate \eqref{eq:estimU'} twice in the proof of our
main result. It will be used with $a= 2/n-2$ in the zero toughness
case and $a= 1/2 - n/2$ in the finite toughness case. We remark that
in both cases, the condition $a>-3/2$ requires that $n <4$.
\end{remark}

\begin{remark}
In terms of Sobolev regularity, we note that \eqref{eq:estimU'} implies that under the assumption of 
Proposition \ref{prop:estim-derU}, $U$ belongs to $H^1(\R)$. 
In particular, $I(U)$ is a function in $L^2(\R)$, and \eqref{lin} implies that $I(U)'\in L^\infty_{loc}(-1,1)$.
We can thus write that $U$ satisfies
\[ I(U)'(x)=f(x) \quad \mbox{ a.e.  in } (-1,1).\]
\end{remark}
\begin{proof}[Proof of Proposition \ref{prop:estim-derU}]
First of all, we note that it is enough to consider $x$ close to $1$
(or $-1$). So we will always assume that $\frac{3}{4} \leq x < 1$.
Using the fact that $z\mapsto f(z)$ and $z\mapsto g(x,z)$ are odd, we
can write
\[ U'(x) = 2 \int_{0}^1 \frac{\pa g}{\pa x}(x,z) f(z)\, dz \]
where we recall that $\frac{\pa g}{\pa x}(x,y)$ is given by \eqref{eq:xder}.

In order to get a bound on $U'(x)$, we first write
\begin{align}
U'(x) &  =2 \int_{0}^{\frac12}
\frac{\pa g}{\pa x}(x,z)f(z)\, dz+2 \int_{1/2}^{1}
  \frac{\pa g}{\pa x}(x,z) f(z)\, dz \nonumber  \\
  & = I_1+I_2.\label{eq:I1I2}
\end{align}

To bound the first integral, we  use \eqref{eq:xder} 
which gives
\[ I_1=- \frac1{\pi}  \int_{0}^{\frac12}
\argsinh\left( \frac{2 xz \sqrt{(1-x^2)(1-z^2)}}{|x^2-z^2|}
\right)f(z)\, dz\] 
and using the fact that $f$ is bounded in $(0,1/2)$, that $\argsinh u
\leq u$ for $u\geq 0$ and that $x-z\geq 1/4$, we deduce
\begin{align}
 |I_1 |&  \leq C \int_{0}^{\frac12}
  \frac{2 xz
  \sqrt{(1-x^2)(1-z^2)}}{|x^2-z^2|}  \, dz\nonumber \\
  &  \leq C \int_{0}^{\frac12}
  \sqrt{ 1-z^2} \, dz  \sqrt{1-x^2} \nonumber \\
  & \leq C \sqrt{1-x^2}. \label{eq:I1}
  \end{align}

In order to estimate $I_2$, we use \eqref{eq:xder-estim} and
\eqref{hyp-resol} (and the fact that $z>1/2$) to write
\begin{align*}
|I_2| &  \leq\frac1{\pi} \int_{1/2}^{1}
  \ln \left(1 +
\frac{4\sqrt{1-x^2}\sqrt{1-z^2}}{|z^2-x^2|}\right)  (1-z^2)^{a}\, dz\\
&  \leq  \frac1{\pi}\int_{1/2}^{1}
    \ln \left(1 +
\frac{4\sqrt{1-x^2}\sqrt{1-z^2}}{|z^2-x^2|}\right)   (1-z^2)^{a}\, z\, dz
\end{align*}
Now, the change of variables $ u = (1-x^2)^{-1}(1-z^2)$, gives
\begin{equation*}
 |I_2| \le C   (1-x^2)^{a+1} \int_0^{\frac{3/4}{1-x^2}} \ln \left(1
+\frac{4 \sqrt{u}}{|1-u|} \right) u^a du.
\end{equation*} 
We note that the integral 
\[  \int_0^{\infty} \ln \left(1 +\frac{4 \sqrt{u}}{|1-u|} \right) u^a du\]
has an integrable singularity at $u=1$; it is convergent at $u=0$ for
all $a>-\frac 3 2$; it is convergent at $u=\infty$ for all $a<-\frac 1
2$.  In particular, we deduce that
\begin{equation}\label{eq:I2a} 
|I_2 | \leq C (1-x^2)^{a+1} \qquad \mbox{ if } -\frac 3 2 < a < -\frac 1 2.
\end{equation}

When $a \geq  -\frac 1 2$, we find that for $x$ close enough to $1$, we have 
\begin{align*}
 |I_2| & \le C   (1-x^2)^{a+1}\left( 1+  \int_2^{\frac{3/4}{1-x^2}} \ln \left(1
+\frac{4 \sqrt{u}}{|1-u|} \right) u^a du\right) \\
& \le 
C   (1-x^2)^{a+1}\left( 1+  \int_2^{\frac{3/4}{1-x^2}}  u^{a-\frac12} du\right). 
\end{align*}
When $a>-\frac{1}{2}$, this implies
\begin{align} 
|I_2| & \leq C   (1-x^2)^{a+1} \left( 1+  (1-x^2) ^{-a-\frac1 2}\right) \nonumber\\
&  \leq C  \sqrt{1-x^2}  .\label{eq:I2b}
\end{align}
While when  $a=-\frac{1}{2}$, we get
\begin{align} 
|I_2| & \leq C   (1-x^2)^{\frac12} \left( 1+  |\ln(1-x^2)| \right)\nonumber\\
&   \leq C  \sqrt{1-x^2} |\ln(1-x^2)| .\label{eq:I2c}
\end{align}

Putting together \eqref{eq:I1I2}, \eqref{eq:I1}, \eqref{eq:I2a}, \eqref{eq:I2b}, \eqref{eq:I2c}, we deduce
\[|U'(x)| \leq \left\{
\begin{array}{ll}
 C (1-x^2)^{a+1} \qquad&  \mbox{ if } -\frac 3 2 < a < -\frac 1 2\\
 C  \sqrt{1-x^2} |\ln(1-x^2)|  \qquad & \mbox{ if }  a = -\frac 1 2\\
 C  \sqrt{1-x^2} \qquad & \mbox{ if }  a> -\frac 1 2
\end{array}
\right.\]
which gives the result.
\end{proof}

To conclude this subsection concerning the linear equation
\eqref{eq:linearIU}, we are going to prove that we can improve
estimate \eqref{eq:estimU'} and derive the precise asymptotic behavior
of $U'(x)$ when $f(z)$ has a particular form.
\begin{proposition}\label{prop:equivU'}
Assume that 
\[f(z) = z(1-z^2)^a h(z)\]
where $h(z)\geq 0$ is a bounded even function on $(-1,1)$ and
$a>-\frac32$. If  $a\leq-\frac12$, we further assume that $h(1) =
\lim_{x \to  1} h(x)$ exists. 

Then the function $U$ defined by \eqref{def:U}   satisfies
\begin{equation}\label{eq:equivU'} 
U'(x) =
 \left\{ 
\begin{array}{ll}
-C_0 (1-x^2)^{a+1 } + o((1-x^2)^{a+1 })& \mbox{ if } -\frac32 < a< - \frac{1}{2} \\
-C_0 (1-x^2)^{\frac 1  2} \ln\left(\frac{1}{(1-x^2)}\right) +\mathcal O ((1-x^2)^{\frac 1 2 }) & \mbox{ if } a= - \frac{1}{2}\\
-C_0 (1-x^2)^{\frac 1  2} +\mathcal O ((1-x^2)^{a+1 }) & \mbox{ if } a> - \frac{1}{2}.
\end{array}
\right.
\end{equation}
where the constant $C_0$ is given by
\begin{equation}\label{eq:C0formula}
C_0 = \left\{
\begin{array}{ll}
c_a h(1) & \mbox{ when } -\frac 3 2< a\leq -\frac 1 2\\
 \frac1{2\pi} \int_{0}^{1}
   2 f(\sqrt{1-v}) v^{-1/2} dv & \mbox{ when } a> -\frac 1 2
\end{array}
\right.
\end{equation}
for some constant $c_a$ depending only on $a$.
\end{proposition}
\begin{proof}
We recall the  formula (using the fact that $z\mapsto f(z)$ is odd and the formula \eqref{eq:xder}):
\begin{align*} 
U'(x) &  = 2 \int_0^1 \frac{\pa g}{\pa x} (x,z) f(z)\, dz \\
& = - \frac1{\pi}  \int_{0}^{1}
\argsinh\left( \frac{2 xz  \sqrt{(1-x^2)(1-z^2)}}{|x^2-z^2|} \right)(1-z^2)^a h(z) z\, dz.
\end{align*}
The change of variables $ u = (1-x^2)^{-1}(1-z^2)$ yields
\[\frac{U'(x)}{(1-x^2)^{a+1}} = - \frac1{2\pi}
\int_{0}^{\frac{1}{1-x^2}} \Theta (x,u) du \]
where the integrand $\Theta (x,u)$ is given by
\[ \Theta (x,u) = \argsinh\left( \frac{2 x\sqrt{1-(1-x^2)u}  \sqrt{u}}{|1-u|} \right)u^a h(\sqrt{1-(1-x^2)u}).\]
Note that $\Theta(x,u)$ is bounded (uniformly in $x$) by
\[ C ||h||_\infty \argsinh\left( \frac{2  \sqrt{u}}{|1-u|} \right)u^a\]
which is integrable on $(0,\infty)$ provided $-\frac32<a<-\frac1 2$.
So Lebesgue dominated convergence theorem implies
\[ \lim_{x\to 1} \int_{0}^{\frac{1}{1-x^2}} \Theta (x,u) du =\int_0^{+\infty} \Theta (1,u) du   =  \int_{0}^{\infty}
\argsinh\left( \frac{2  \sqrt{u}}{|1-u|} \right)u^a h(1) du\]
which gives \eqref{eq:equivU'} and \eqref{eq:C0formula} in the case
$-\frac32<a<-\frac1 2$ (and we see that this limit is strictly
positive as soon as $h(1)>0$).
\medskip

When $a\geq -\frac1 2$, we write, for $\frac1{1-x^2}\geq 2$: 
\begin{align*}
\frac{U'(x)}{\sqrt{1-x^2}}& = - \frac1{2\pi} (1-x^2)^{a+\frac12}
\int_{0}^{\frac{1}{1-x^2}} \Theta (x,u) du\\ & = - \frac1{2\pi}
(1-x^2)^{a+\frac12} \int_{0}^{2}\Theta (x,u) du - \frac1{2\pi}
(1-x^2)^{a+\frac12} \int_{2}^{\frac{1}{1-x^2}}\Theta (x,u)
du \smallskip \\ 
& = I_1+I_2.
\end{align*} 
The first term satisfies
\[ |I_1|\leq  C (1-x^2)^{a+\frac12}||h||_\infty \int_{0}^{2}
\argsinh\left( \frac { \sqrt{u}}{|1-u|} \right)u^a  du \]
and so
\begin{align*}
\lim_{x\to1}|I_1| = 0  & \mbox{ if } a>-\frac 1 2 \\
|I_1|  \leq C&  \mbox{ if } a=-\frac 1 2.
\end{align*}
For the second term, we recall that $|\argsinh(w)-w|\leq Cw^3$, and so
for all $2\leq u\leq\frac1{1-x^2} $, we have
\[\left|\argsinh\left( \frac{2 x\sqrt{1-(1-x^2)u}  \sqrt{u}}{|1-u|}
\right) - \frac{2 x\sqrt{1-(1-x^2)u}  
\sqrt{u}}{|1-u|} \right| \leq \frac{C}{u^\frac32},\]
which also yields
\[\left|\argsinh\left( \frac{2 x\sqrt{1-(1-x^2)u}  \sqrt{u}}{|1-u|}
\right) - \frac{2 x\sqrt{1-(1-x^2)u}}{  \sqrt{u}}\right| \leq \frac{C}{u^\frac32}.\]
We deduce
\begin{align*}
I_2 &  = - \frac1{2\pi} (1-x^2)^{a+\frac12} \int_{2}^{\frac{1}{1-x^2}}
   2 x\sqrt{1-(1-x^2)u}
    h(\sqrt{1-(1-x^2)u}) u^{a-\frac1 2} du +R\\
   & = - \frac1{2\pi} \int_{2(1-x^2)}^{1}
   2 x
   \sqrt{1-v}h(\sqrt{1-v}) v^{a-\frac1 2}  dv +R
\end{align*}
where
\begin{align*}
R & \leq C||h||_\infty (1-x^2)^{a+\frac12} \int_{2}^{\frac{1}{1-x^2}}
u^{a-\frac3 2} du\\
&= \mathcal O((1-x^2)+(1-x^2)^{a+\frac1 2}). 
\end{align*}
 
 When $a>-\frac1 2$, we deduce that 
\[\lim_{x\to 1} I_2 =   - \frac1{2\pi} \int_{0}^{1} 2 
   \sqrt{1-v}h(\sqrt{1-v}) v^{a-\frac1 2}  dv,\]
which implies \eqref{eq:equivU'} and \eqref{eq:C0formula} in that case
(note that $\sqrt{1-v}h(\sqrt{1-v}) v^{a} = f(\sqrt{1-v})$).
\medskip

When $a=-\frac1 2 $, we use L'Hospital's Rule  to prove that
\[\lim_{x\to 1-} \frac {I_2}{\ln(1-x^2)} = \frac{1}{\pi} h(1)\]
which gives \eqref{eq:equivU'} and \eqref{eq:C0formula} in the case
$a=-\frac1 2 $ and completes the proof.
\end{proof}

\section{Proof of the main result}\label{sec:main}

We are now ready to prove our main result, that is the existence of
self-similar solutions for
\eqref{eq:ftf}-\eqref{eq:bc}-\eqref{eq:tough}.  As shown in
Subsection~\ref{sec:notough}, the proof of Theorem~\ref{thm:main} reduces
to the proving Propositions~\ref{prop:1} and \ref{prop:2}, which is the
goal of this section.

\subsection{The zero toughness case: Proof of Proposition \ref{prop:1}}
In this section, we will prove Proposition \ref{prop:1}, that is the
existence of a solution $U(x)$ of \eqref{eq:self} satisfying
\eqref{eq:Uequiv}.

\begin{remark}
We already mentioned that for $n=1$, the function $U(x) = \frac49
(1-x^2)_+^{\frac32}$ is a solution of \eqref{eq:self} (see
Lemma~\ref{lem:casen1} in Appendix). In the sequel, we will thus
always assume that $n\in(1,4)$. 
\end{remark}

We recall that, using the Green function $g(x,z)$ introduced in
Section \ref{sec:Green}, we can rewrite, formally at least, equation
\eqref{eq:self} as the following integral equality:
\begin{equation}\label{eq:selfint}
U(x) =  \int_{-1}^1 z g(x,z) \left(U(z)\right)^{1-n} dz .
\end{equation}
The fact that a solution of \eqref{eq:selfint} actually solves
\eqref{eq:self} will follow from Proposition~\ref{prop:linear-g} once
we have established appropriate estimates on $U$ (more precisely, we
will need to control the behavior of $\left(U(z)\right)^{1-n}$ near
$z=\pm 1$).

Now, we will find a solution of \eqref{eq:selfint} by a fixed point
argument.  However, when $n>1$, the integrand is singular whenever
$U(z)=0$, so we first construct approximate solutions of
\eqref{eq:selfint} as follows:
\begin{lemma}[Construction of an approximate solution]\label{lem:approx}
For any $n \in (1,4)$ and for all $k \in \N$, there exists a
continuous function $U_k\colon \R \to \R $ such that
\begin{equation}\label{eq:approx-integral}
\left\{ \begin{array}{ll}
\displaystyle U_k(x) =  \int_{-1}^1 z g(x,z) \left(\frac1k +U_k(z)\right)^{1-n} dz & \text{ for
} x \in (-1,1)\\[5pt]
\displaystyle U_k (x) = 0 & \text{ for } x \notin (-1,1).
\end{array}\right.
\end{equation}
Furthermore, $U_k$ is non-negative in $\R$ and is $C^1$ in $(-1,1)$.
\end{lemma}
\begin{proof}
We will first construct the solution in the interval $[-1,1]$ (we then
extend $U_k$ by zero outside $[-1,1]$). For that, we consider the
following closed convex set of $C([-1,1])$
\[ S = \{ V \in C ([-1,1]): V(x) = V(-x), 0 \le V \le A  \text{ in } [-1,1]\} \] 
(for a positive number $A>0$ to be fixed later)
and the operator $\mathcal{T}\colon S \to C([-1,1])$ which maps
$V \in S$ to the function
\[ U (x) =  \int_{-1}^1 z g(x,z) \left(\frac1k +V(z)\right)^{1-n} dz. \]
Proposition \ref{prop:g}-\ref{pos}) implies that $z\mapsto z g(x,z)$ is even and positive on $(-1,1)$ for all $x\in (-1,1)$, so
\[U(x)\geq 0 \quad \text{ in } (-1,1).\]
Proposition \ref{prop:g}-\ref{pos}) also implies that 
$x\mapsto U(x)$ is even. 
Next, Proposition~\ref{prop:linear-g} 
and the fact that $\left(\frac1k +V(z)\right)^{1-n} \leq  k^{n-1}$
imply
that $U\in C^1(-1,1)$ and satisfies (see \eqref{eq:bdU})
\begin{equation}\label{eq:bdU2} 
|U(x)| \leq C k^{n-1}  \sqrt{1-x^2} \qquad \quad\forall x\in(-1,1).
\end{equation}
Finally,
the bound \eqref{estim:xzG} gives  in particular $|g(x,z)|\leq 1$ for all
$x,\,z\in(-1,1)$. Hence
\[0\leq U(x)\leq 2 k^{n-1}\quad\mbox{ for all $x\in(-1,1)$}.\]
Choosing $A= 2 k^{n-1}$, we deduce that
\[ \mathcal{T}(S) \subset S.\]

Moreover, Proposition~\ref{prop:estim-derU} (see \eqref{eq:estimU'}
with $a=0$) implies
\[ |U'(x)|\leq C k^{n-1} \sqrt{1-x^2} \quad \forall x\in(-1,1)\] 
and so $\mathcal{T}(S)$ is equi-Lipschitz continuous.  Using
Ascoli-Arzel\`a's theorem, we deduce that $\mathcal{T}(S)$ is a
compact subset of $C ([-1,1])$.  Finally, using once again the fact
that $|g(x,z)| \le 1$ together with Lebesgue dominated convergence
Theorem, it is easy to show that $\mathcal{T}$ is a continuous
operator. We can thus use Schauder's fixed point Theorem and deduce
that $\mathcal{T}$ has a fixed point $U_k$.  We can now define
$U_k(x)=0$ for $x\notin[-1,1]$. Using \eqref{eq:bdU2}, the resulting
function is indeed continuous in $\R$.
\end{proof}

In order to pass to the limit $k\to \infty$, we now need to derive
some estimates on $U_k$ which do not depend on the parameter $k$.
\begin{lemma}[Uniform estimates]
For  $n \in (1,4)$, there exists $C>0$ such that for all $k \in \N$, the
function $U_k$ constructed in Lemma~\ref{lem:approx} satisfies, for all $x\in(-1,1)$:
\begin{eqnarray}
\label{uk:nondeg} \frac1k + U_k (x) &\ge& C^{-1} (1-x^2)^{\frac2n} \\
\label{uk:bound} |U_k(x)| & \leq &  C \sqrt{1-x^2}\\
\label{uk':bound} |U_k'(x)| &\le& C \begin{cases}
(1-x^2)^{\frac2{n} -1} & \text{ if } n \in (\frac43,4) \\
\sqrt{1-x^2} \ln \frac{1}{\sqrt{1-x^2}} & \text{ if } n = \frac43 \\
\sqrt{1-x^2} & \text{ if } n \in [1, \frac43).
\end{cases} 
\end{eqnarray}
\end{lemma}
\begin{proof}
In view of \eqref{eq:xder}, the function $x\mapsto z g(x,z)$ is
decreasing on the interval $[0,1]$, for all $z\in(-1,1)$. The
definition of $U_k$, \eqref{eq:approx-integral}, thus implies that
$x\mapsto U_k(x)$ is non-increasing on the interval $[0,1]$.  Using
\eqref{non-deg}, we deduce that for $x \in (0,1)$ we have
\[ U_k (x) \ge \left(\frac1k +U_k (x)\right)^{1-n} \int_x^1  z g(x,z) dz \ge 
C \left(\frac1k +U_k (x)\right)^{1-n} (1-x^2)^2 \]
which yields \eqref{uk:nondeg}, and, in turns, gives
\[f(z)=z \left(\frac1k + U_k (z)\right)^{1-n} \leq \left(1-z^2\right)^{\frac{2}{n}-2}.\]
We note that for $n<4$, we have $a=\frac{2}{n}-2>-3/2$, so Proposition 
\ref{prop:linear-g} gives \eqref{uk:bound} 
and Proposition \ref{prop:estim-derU} (note that $z\mapsto f(z)$ is odd) implies
\eqref{uk':bound}.
\end{proof}

We can now pass to the limit $k \to \infty$ in
\eqref{eq:approx-integral} and complete the proof of
Proposition~\ref{prop:1}.
\begin{proof}[Proof of Proposition~\ref{prop:1}]
Thanks to Estimates \eqref{uk:bound}-\eqref{uk':bound}, 
Ascoli-Arzel\`a's Theorem implies that there exists a subsequence,
denoted $U_p$, of $U_k$ and a function $U(x)$ defined in $(-1,1)$ such
that $U_p(x) \longrightarrow U(x)$ as $p\to\infty$, locally uniformly
in $(-1,1)$.  Moreover, \eqref{uk:bound} implies that
\[  |U(x)| \leq   C \sqrt{1-x^2} \quad \mbox{ in } (-1,1) ,\]
so we can define $U(x)=0$ for $x\notin (-1,1)$ and get a
continuous function in $\R$.
Finally, \eqref{uk:nondeg} implies
\begin{equation}\label{eq:Unondeg} 
U(x)\geq C^{-1} (1-x^2)^{\frac2n} \quad \mbox{ in } (-1,1).
\end{equation}

Furthermore, we note that the sequence of functions 
\[f_p(x)=x\left(\frac1p+U_p(x)\right)^{1-n}\] 
converges locally uniformly to
$f(x)=x U(x)^{1-n}$ and satisfies (using \eqref{uk:nondeg})
\[ |f_p(x)|\leq C^{-1} (1-x^2)^{\frac2n -2}.\] 
Since $\frac2n -2 >-\frac 3 2$ for $n\in(1,4)$, and in view of
Proposition \ref{prop:linear-g}, we can pass to the limit in
\eqref{eq:approx-integral} and deduce that $U$ satisfies
\eqref{eq:selfint}, that is
\[U(x) =  \int_{-1}^1 g(x,z) \left(U(z)\right)^{1-n} z dz .\]

Proposition \ref{prop:linear-g} also implies 
that $U$ is in $C^1(-1,1) $ and solves
\begin{align*}
I(U)' = x U^{1-n} & \text{ in  }   \mathcal{D}'((-1,1)), \\
U = 0 & \text{ for } x \notin (-1,1). 
\end{align*}
Note that this implies in particular for that $I(U)'$ in $L^\infty_{loc}(-1,1)$ and that
\[ U^{n} I(U)'=xU \quad \mbox{ for all } x\in (-1,1).\]
\medskip

It remain to prove \eqref{eq:Uequiv} which now follows from
Proposition \ref{prop:equivU'}.  Indeed $U$ is given by
\[U(x)=\int_{-1}^1 g(x,z) f(z)dz\]
with
\[f(z)=z \left(U(z)\right)^{1-n} =z
\left(\frac{U(z)}{(1-z^2)^{2/n}}\right)^{1-n} (1-z^2)^{\frac2n -2}.\]
We can thus apply Proposition \ref{prop:equivU'} with
$h(z)=\left(\frac{U(z)}{(1-z^2)^{2/n}}\right)^{1-n} $ and
$a=\frac{2}{n}-2$ (note that the function $h(z)$ is in particular
non-negative, bounded and even).  We deduce
\begin{equation}\label{eq:U'as} 
U'(x)= 
 \left\{
\begin{array}{ll}
-C_0(1-x^2)^{\frac{1}{2}} + \mathcal O (|1-x^2|^{\frac 2  n -1})& \mbox{ if } n \in [1,\frac{4}{3})\\
-C_0 (1-x^2)^{\frac 1  2} \ln\left(\frac{1}{(1-x^2)}\right) +\mathcal O ((1-x^2)^{\frac 1 2 })  & \mbox{ if } n= \frac{4}{3}\\
-C_0(1-x^2)^{\frac{2}{n}-1} + o (|1-x^2|^{\frac 2  n -1})& \mbox{ if } n \in  (\frac{4}{3},4) 
\end{array}
\right.
\end{equation}
and \eqref{eq:Uequiv} follows (using the fact that $U(\pm 1) = 0$).
Note in particular that \eqref{eq:Unondeg} implies that $C_0\neq 0$ in
the case $n \in (\frac{4}{3},4)$, while formula \eqref{eq:C0formula}
gives $C_0\neq 0$ in the case $n \in [1,\frac{4}{3})$.
\bigskip

In the critical case $n=\frac{4}{3}$, however, we can show that $C_0=0$.
Indeed, in that case, we have 
\[ h(z)= \left(\frac{U(z)}{(1-z^2)^{3/2}}\right)^{-\frac{1}{3}} \]
and so \eqref{eq:U'as} implies that $h(1)=0$ and in turn, formula
\eqref{eq:C0formula} gives $C_0=0$.
We thus need to work some more to derive the correct behavior as $x\to \pm 1$, namely
$$U'(x)\sim -C_0 (1-x^2)^{\frac 1  2} \ln\left(\frac{1}{(1-x^2)}\right)^{3/4}.$$
The interested reader will find the proof of this fact in Appendix \ref{app:asymp}.
\end{proof}

\subsection{The finite toughness case: Proof of Proposition \ref{prop:2}}\label{sec:finite}

We now consider the case of positive toughness $K\neq 0$.  As shown in
Subsection~\ref{sec:tough}, the proof of Theorem \ref{thm:main} in this
case is equivalent to proving Proposition \ref{prop:2}, that is the
existence of a solution $U(x)$ to equation \eqref{eq:self-bis}
satisfying \eqref{eq:Uasympttough}.

We recall (see Section \ref{sec:strategy}) that equation
\eqref{eq:self-bis} can be (formally) written as the following
integral equality:
\begin{equation}\label{eq:toughint}
 U (x) = \int_{-1}^1 g(x,z) U^{-n} \left( z U(z)+ \frac32 \mathcal{U}(z) \right) 
  dz + K \sqrt{1-x^2}, \quad x \in [-1,1]
  \end{equation}
with 
\[\mathcal{U}_k (z) = \left\{\begin{array}{ll}\ds  \int_z^1 U_k (y) \, dy& \mbox{ for } z>0\\[8pt]
\ds -\int_{-1}^z U_k (y) \, dy& \mbox{ for } z<0.\end{array}\right.\]
As we did in the zero toughness case, we will solve
\eqref{eq:toughint} by a fixed point argument.  But we first need to
solve an approximate problem to avoid the singularity in
\eqref{eq:toughint} when $U=0$.  Because of the term $\mathcal V(z)$,
the approximation that we use here is slightly different from that of the previous section:
\begin{lemma}[Construction of an approximate solution]\label{lem:approx-bis}
For all $k \in \N$, there exists a continuous function $U_k\colon
[-1,1] \to ]0,+\infty[$ such that for all $x \in (-1,1)$, 
\begin{equation}\label{eq:approx-integral-bis}
\displaystyle U_k(x) =\frac1k+ \int_{-1}^1  g(x,z) \left( U_k(z)\right)^{-n} \left(z U_k (z) 
+ \frac32 \mathcal{U}_k (z)\right)  dz +  K \sqrt{\frac 1 k+1-x^2}.
\end{equation}
Furthermore, $U_k$ is non-negative in $\R$ and is $C^1$ in $(-1,1)$.
\end{lemma}
\begin{proof}
The proof follows that of Lemma \ref{lem:approx} with minor
modifications.  We consider the closed convex set of $C([-1,1])$
\begin{multline*}
 S = \{ V \in C ([-1,1]): \frac{1}{k} \le V \le A  \text{ in } [-1,1], \\
 x\mapsto V(x) \text{ even in [-1,1] and non-increasing in $[0,1]$} \} 
\end{multline*}
(for a positive number $A>0$ to be fixed later) and the operator
$\mathcal{T}\colon S \to C([-1,1])$ which maps $V \in S$ to the
function
\begin{align*}
U(x) &=\frac1k+   \int_{-1}^1  g(x,z) \left( V(z)\right)^{-n} \left(z V (z) 
+ \frac32 \mathcal{V} (z)\right)  dz  + K \sqrt{\frac 1 k+1-x^2}. 
\end{align*}
Note that since $x\mapsto g(x,z)$ is even (see
Proposition~\ref{prop:g}-\ref{pos})), so is the function $U$, and
using the fact that $z\mapsto g(x,z)$ is odd, we can rewrite this
equality as
\begin{align*}
U(x) &=\frac1k+ 2  \int_{0}^1  g(x,z) \left( V(z)\right)^{-n} \left(z V (z) 
+ \frac32 \mathcal{V} (z)\right)  dz  + K \sqrt{\frac 1 k+1-x^2}. 
\end{align*}
Proposition \ref{prop:g}-\ref{pos}) implies that the integrand is
non-negative in $(0,1)$, and so it is readily seen that
\[ U(x)\geq \frac 1k.\]
Using now Proposition \ref{prop:g}-\ref{inc}) implies that $x\mapsto U(x)$
is non-increasing on $(0,1)$.  Next, we note that for $V\in S$, we
have
\[ \mathcal{V} (z) = \int_z^1 V(y) \, dy \le (1-z) V(z)\]
and so, for $z \in (0,1)$, 
\begin{equation}\label{vcal-et-v}
 z V(z) + \frac32 \mathcal{V} (z) \le \frac32 V(z).
\end{equation}
We thus have (using  \eqref{estim:xzG})
\begin{align*} 
U(x) & \leq \frac1k+  3\int_{0}^1  g(x,z) \left( V(z)\right)^{1-n}   dz  + K \sqrt{1-x^2}\\
& \leq  \frac1k + 3 k^{n-1}+K\sqrt{1+\frac 1 k},
\end{align*}
so we choose
\[ A= \frac1k + 3 k^{n-1}+K\sqrt{1+\frac 1 k}\]
and we deduce
\[ \mathcal{T}(S) \subset S.\]

Moreover, Proposition~\ref{prop:linear-g} (see \eqref{eq:estimU'} with
$a=0$) implies that $U'$ is $C^1$ in $(-1,1)$ and
\[ |U'(x)|\leq C(k,A) \sqrt{1-x^2} + 2 K \left(\frac 1 k + 1- x^2\right)^{-\frac 1 2} \quad \forall x\in(-1,1)\]
and so $\mathcal{T}(S)$ is equi-Lipschitz continuous.  Hence
$\mathcal{T}(S)$ is compact.  Finally, using once again the fact that
$|g(x,z)| \le 1$ together with Lebesgue dominated convergence Theorem,
it is easy to show that $\mathcal{T}$ is a continuous operator. We can
thus use Schauder's fixed point Theorem and deduce that $\mathcal{T}$
has a fixed point $U_k$.
\end{proof}

We then derive uniform  (with respect to $k$) estimates for these
approximate solutions.
\begin{lemma}[Uniform estimates]\label{lem:unif}
Let $K>0$ and assume $n \in [1,4)$. There exists a constant $C>0$ depending only on $n$ such that for
  all $k \in \N$, the function $U_k$ constructed in
  Lemma~\ref{lem:approx-bis} satisfies
\begin{eqnarray}
\label{uk:nondeg-bis}  U_k (x) &\ge& K (1-x^2)^{\frac12} \qquad \mbox{ for all } x\in(-1,1) \\
\label{uk:bound-bis} |U_k(x)| &\le& \frac C k+ C(K^{1-n}+K) (1-x^2)^{\frac12} \qquad \mbox{ for all } x\in(-1,1)\\
\label{uk':bound-bis} |U_k'(x)| &\le& C (K^{1-n}+K)  (1-x^2)^{-\frac12} \qquad \mbox{ for all } x\in(-1,1).
\end{eqnarray}
\end{lemma}
\begin{proof}
Estimate~\eqref{uk:nondeg-bis} follows immediately from 
\eqref{eq:approx-integral-bis} (note that the first two terms in the right hand side are non-negative).  
Next, we note (using \eqref{vcal-et-v}),  that  the odd  function
\[ f_k(z)=\left( U_k(z)\right)^{-n} \left(z U_k (z) + \frac32 \mathcal{U}_k (z)\right)\]
satisfies
\[ f_k(z)\leq \frac32 U_k^{1-n} (z) \leq C K^{1-n} (1-z^2)^{\frac{1-n}{2}} .\]
In particular, $f_k$ satisfies the condition of Propositions
\ref{prop:linear-g} and \ref{prop:estim-derU} with $a = \frac{1-n}2 >
-\frac32$ provided $n <4$.  Proposition \ref{prop:linear-g} now implies
\eqref{uk:bound-bis}, and Proposition~\ref{prop:estim-derU} gives
\begin{align*}
|U_k'(x)| 
&\le C K^{1-n} F(\sqrt{1-x^2}) +  K\frac{1}{\sqrt{1-x^2}}\\
& \le C(K^{1-n}+K)\frac{1}{\sqrt{1-x^2}} 
\end{align*}
(recall that $F$ is given by \eqref{eq:F}), which is exactly \eqref{uk':bound-bis}.
\end{proof}

We can now pass to the limit $k\to \infty$ and complete the proof of
Proposition~\ref{prop:2}.
\begin{proof}[Proof of Proposition \ref{prop:2}]
Estimates from Lemma~\ref{lem:unif} together with the fact that $U_k
(\pm 1) = \frac1k$ implies that we can extract a subsequence $U_p$ which
converges locally uniformly in $(-1,1)$ towards a continuous function
$U$ which vanishes at $\pm1$.

First, we can pass to the limit in \eqref{eq:approx-integral-bis} by using 
\eqref{uk:nondeg-bis}, \eqref{estim:xzG}
and Legesgue dominated convergence theorem (note that $1-\frac n2>-1$ when $n<4$).
We deduce that $U$ satisfies \eqref{eq:toughint}, and 
Proposition~\ref{prop:linear-g}
implies that $U$ solves 
\[ 
I(U)' = U^{-n} \left(z U + \frac32 \mathcal{U} \right) \mbox{ in } (-1,1). 
\]

In order to study the behavior of $U$ near $x=\pm1$, we write
\[ U(x)= W(x) + K\sqrt{1-x^2}\]
with 
\begin{align*}
W(x) & =\int_{-1}^1 g(x,z) U^{-n} \left( z U(z)+ \frac32 \mathcal{U}(z) \right)   dz\\
  & = \int_{-1}^1 g(x,z) f(z)  dz
\end{align*}
where (proceeding as above), we see that $f(z)$ satisfies
\[f(z)\leq \frac32 U_k^{1-n} (z) \leq C K^{1-n} (1-z^2)^{\frac{1-n}{2}}. \]
 Proposition~\ref{prop:estim-derU} with $a=\frac{1-n}{2}$ (see also \eqref{eq:ffrU}) thus implies that there exists a constant $C$ such that
\[| W(x)|\leq \left\{\begin{array}{ll}
C (1-x^2)^{\frac3 2} & \mbox{ if } 1<n<2 \\
C(1-x^2)^{\frac3 2}  \ln\left( \frac{1}{1-x^2}\right) & \mbox{ if } n=2\\
C(1-x^2)^{\frac{5-n}2} & \mbox{ if } 2<n<4
\end{array}
\right.\]
and the proof is now complete.
\end{proof}

\appendix

\section{Some technical results}

\subsection{An explicit solution}

\begin{lemma}[The special case $n=1$]\label{lem:casen1}
For $U (x) = \frac49 (1-x^2)_+^{\frac32}$, we have 
\[ I (U) ' = x \text{ for } x \in [-1,1].\]
\end{lemma}
\begin{proof}
We first compute the Riesz potential $\mathcal{I}_\beta (U)$ for $\beta \in
(0,1)$ by using \cite[Lemma~4.1]{bik13} and get
\[ \mathcal{I}_\beta (U) = \frac49 C_{3,\beta,1} \times {}_2 F_1 \left(\frac{1-\beta}2,-2;1;x^2 \right).\]
Hence differentiating and using the fact that ${}_2 F_1 (a,-1;c;z)= 1 -\frac{a}c z$, we get 
\begin{align*}
 (\mathcal{I}_\beta (U))' & = -\frac49 C_{3,\beta,1} (1-\beta) \times {}_2 F_1 \left(\frac{3-\beta}2,-1;2;x^2 \right) (2x) \\
 & =  D_\beta (\frac{3-\beta}4 x^2 -1) 2x = D_\beta \frac{3-\beta}2 x^3 - 2 D_\beta x
 \end{align*}
 where 
 \[ D_\beta = \frac49 \cdot\frac{(1-\beta)\Gamma 
(\frac{5}{2})\Gamma(\frac{1-\beta}{2})}{2^\beta\Gamma (\frac{1}{2})\Gamma(1+\frac{3+\beta}{2})} 
 = \frac49 \cdot \frac{2^{1-\beta}\Gamma (\frac{5}{2})
\Gamma(\frac{3-\beta}{2})}{\Gamma (\frac{1}{2})\Gamma(1+\frac{3+\beta}{2})}. \]
 Then 
\[ D_\beta \to \frac49 \cdot \frac{\Gamma
  (\frac{5}{2})\Gamma(1)}{\Gamma 
(\frac{1}{2})\Gamma(3)} = \frac16.\]
Hence 
\[ (\mathcal{I}_1 (U))' = \frac{1}6 x^3 - \frac13 x\]
and 
\[ (I(U))' = (\mathcal{I}_1(U))''' =  x.\]
\end{proof}

\subsection{Proof of Lemma \ref{lem:asymp2}}

\begin{proof}[Proof of Lemma \ref{lem:asymp2}]
When $x\geq z$ or $x\leq -z$, the $y\mapsto G(x,y)$ has no
singularities in the interval $(-z,z)$, and a simple integration by
parts yields
\begin{align*} 
 \int_{-z}^z G(x,y)\, dy & = (y-x)G(x,y)\Big|_{y=-z}^{y=z} -
 \int_{-z}^z (y-x)\pa_yG(x,y)\, dy\\ & = (z-x)G(x,z) +(z+x)G(x,-z) +
 \int_{-z}^z (x-y)\pa_yG(x,y)\, dy.
\end{align*}
Lemma~\ref{lem:derivatives-G} implies
\[ \int_{-z}^z (x-y)\pa_yG(x,y)\, dy 
= \frac 1 \pi \sqrt{1-x^2}[\arcsin(z)-\arcsin(-z)] = \frac 2\pi \sqrt{1-x^2}\arcsin(z)\]
and \eqref{eq:GG1} follows.

When $-z\leq x\leq z$, we need to split the integral:
\[ \int_{-z}^z G(x,y)\, dy = \int_{-z}^x G(x,y)\, dy+\int_{x}^z G(x,y)\, dy.\]
We then proceed as before to evaluate those integrals, after noticing
that the function $y\mapsto (y-x)G(x,y)$ vanishes for $y=x$:
\begin{eqnarray*} 
 \int_{-z}^x G(x,y)\, dy & = & (y-x)G(x,y)\Big|_{y=-z}^{y=x} - \int_{-z}^x (y-x)\pa_yG(x,y)\, dy\\
 & =&  (z+x)G(x,-z) + \int_{-z}^x (x-y)\pa_yG(x,y)\, dy \\
 & =&  (z+x)G(x,-z) + \frac 1 \pi \int_{-z}^x\frac{\sqrt{1-x^2}}{\sqrt{1-y^2}} \, dy \\
 & =&  (z+x)G(x,-z) +\frac 1 \pi  \sqrt{1-x^2}\big[\arcsin(x)+\arcsin(z)\big]  
\end{eqnarray*}
and
\begin{eqnarray*}
 \int_{x}^z G(x,y)\, dy & = & (y-x)G(x,y)\Big|_{y=x}^{y=z} - \int_{-z}^x (y-x)\pa_yG(x,y)\, dy\\
 & =&  (z-x)G(x,z) - \int_{x}^z (y-x)\pa_yG(x,y)\, dy \\
  & =&  (z-x)G(x,z) +\frac 1 \pi \int_{x}^z \frac{\sqrt{1-x^2}}{\sqrt{1-y^2}} \, dy \\
 & =&  (z-x)G(x,z) + \frac 1 \pi \sqrt{1-x^2}\big[\arcsin(z)-\arcsin(x)\big]  .
\end{eqnarray*}
The result follows.
\end{proof}

\section{Boundary behavior in the critical case} \label{app:asymp}
In this section, we complete the proof of Proposition~\ref{prop:1} by
deriving the boundary behavior of the function $U(x)$ in the critical case $n=\frac 4 3$ when $K=0$.
More precisely, we will show that in that case we have
\begin{equation}\label{eq:appasymp}
 U (x) \sim (2/9\pi)^{\frac34} (1-x^2)^{\frac32} |\ln (1-x^2)|^{\frac34},
 \end{equation}
 when $x\to\pm1$.

Since $U$ is even, it is enough to look at the case $x\to 1$, and  we recall that
\[U(x)=\int_{-1}^1 g(x,z) f(z)dz\]
where
\[f(z)=z \left(U(z)\right)^{-\frac 1 3} =z
\left(\frac{U(z)}{(1-z^2)^{3/2}}\right)^{-\frac 1 3} (1-z^2)^{\frac 2 3}.\]
We thus denote 
\[ h(x) = \left( \frac{U(x)}{(1-x^2)^{\frac32}} \right)^{-\frac13}.\]
Inequality \eqref{eq:Unondeg}  and the inequality above imply
\[ 0 < c \le \frac{U(x)}{(1-x^2)^{\frac32}} \le C |\ln (1-x^2)|\]
and so
\[  \frac{c}{|\ln (1-x^2)|^{\frac13}} \le h(x) \le C .\]

Proceeding as  in the proof of Proposition~\ref{prop:equivU'} (in the case $a=-\frac 1 2$), we can  prove the following lemma:
\begin{lemma}\label{lem:hcritical}
In the case $n=\frac{4}{3}$, we have 
$$ \frac{U'(x)}{\sqrt{1-x^2}} =  - \frac1{2\pi} \int_{2(1-x^2)}^{1}
    \frac{h(\sqrt{1-v}) }{v}  dv   + \tilde R(x)$$
    where $\tilde R(x)$ is a bounded function for $x\geq 0$ and 
\[ h(x) = \left( \frac{U(x)}{(1-x^2)^{\frac32}} \right)^{-\frac13}.\]
\end{lemma}
Postponing the proof of this lemma to the end of this section, 
we now define $H(y) = h (\sqrt{1-y})$ and
\[ E (y) = \int_y^1 \frac{H(\tau)}{2 \pi \tau} d \tau .\]
Remark that $E$ is differentiable and satisfies 
\begin{equation}\label{eq:E'}
 E' (y) = - \frac{H(y)}{2 \pi y}.
\end{equation}
We also have 
\[ E (y ) \ge  \int_y^1 \frac{c}{2\pi|\ln (\tau)|^{\frac13}}
\cdot \frac{d\tau}{\tau}\]
and in particular, 
\begin{equation}\label{eq:E+infty}
E(y) \to + \infty \quad \text{ as } y \to 0. 
\end{equation}
\[ U'(x) \sim -(1-x^2)^{\frac12} E(2(1-x^2)) \quad \text{ as } x \to 1^-.\]
We claim that this implies 
\begin{equation}\label{eq:claim}
 U(x) \sim \frac23 (1-x^2)^{\frac32} E (2(1-x^2))\quad \text{ as } x \to 1^-.
\end{equation}
Indeed, 
\[ \int_x^1 U'(y) dy  \sim  -\int_x^1 (1-y^2)^{\frac12} E
(2(1-y^2)) dy \]
implies 
\[ U(x) \sim \int_0^{1-x^2} z^\frac12 E (2z) dz .\]
Integrating by parts, we get,
\begin{align*}
 \int_0^{1-x^2} z^\frac12 E (2z) dz &=  \frac23 \big[ E(2z) 
  z^{\frac32}  \big]_0^{1-x^2} - \frac43 \int_0^{1-x^2} E'(2z)
z^{\frac32} dz \\
& = \frac23 (1-x^2)^\frac32 E (2(1-x^2)) + \frac{2}{3\pi} \int_0^{1-x^2} H(z)
z^{\frac12} dz \\
&  = \frac23 (1-x^2)^\frac32 E (2(1-x^2)) + \mathcal{O} ((1-x^2)^{\frac32})
\end{align*}
(where we used \eqref{eq:E'} and the fact that $h$ is bounded). 
In view of \eqref{eq:E+infty}, it follows
that \eqref{eq:claim} indeed holds true. 
\medskip

Now, Equation~\eqref{eq:claim} implies that the function $H$ satisfies 
\[ H (y) \sim \left(\frac23 E(2y)\right)^{-\frac13}  \quad \text{ as } y \to 0.\]
Furthermore, L'Hospital's rule implies
\[ \lim_{y \to 0} \frac{E(2y)}{E(y)} = \lim_{y \to 0}
\frac{H(2y)}{H(y)} = \left(\lim_{y \to 0} \frac{E(2y)}{E(y)}\right)^{-\frac13}\]
and so
\[ E(2y) \sim E(y) .\]
We can thus write
\[ H (y) \sim \left(\frac23 E(y)\right)^{-\frac13}  \quad \text{ as } y \to 0.\]
In view of \eqref{eq:E'}, this implies
\[ -4 \pi y E'(y) \sim \left(\frac23 E(Ay)\right)^{-\frac13}  \quad \text{ as } y \to 0,\]
or
\[ (E^\frac43)'(y) \sim - \frac1{3 \pi y} \quad \text{ as } y \to 0.\]
This finally gives
\[ E (y) \sim (3/2)^{\frac14} (3 \pi)^{-\frac34} |\ln y|^{\frac34},\]
and \eqref{eq:claim} implies finally 
\[ U (x) \sim (2/9\pi)^{\frac34} (1-x^2)^{\frac32} |\ln (1-x^2)|^{\frac34}.\]
The proof of Proposition~\ref{prop:1} is now complete.

\begin{proof}[Proof of Lemma \ref{lem:hcritical}]
We proceed as in the proof of Proposition~\ref{prop:equivU'} (in the case $a=-\frac 1 2$).
First, we have 
\[\frac{U'(x)}{(1-x^2)^{\frac 1 2}} = - \frac1{2\pi}
\int_{0}^{\frac{1}{1-x^2}} \Theta (x,u) du \]
where the integrand $\Theta (x,u)$ is given by
\[ \Theta (x,u) = \argsinh\left( \frac{2 x\sqrt{1-(1-x^2)u}  \sqrt{u}}{|1-u|} \right)u^{-\frac 1 2 } h(\sqrt{1-(1-x^2)u}).\]
Next, we write, for $\frac{1}{1-x^2}\geq 2$:
\begin{align*}
\frac{U'(x)}{\sqrt{1-x^2}}& = - \frac1{2\pi}  
\int_{0}^{\frac{1}{1-x^2}} \Theta (x,u) du\\ & = - \frac1{2\pi}
  \int_{0}^{2}\Theta (x,u) du - \frac1{2\pi}
 \int_{2}^{\frac{1}{1-x^2}}\Theta (x,u)
du \smallskip \\ 
& = I_1+I_2.
\end{align*} 
where the first term satisfies
\[ |I_1|\leq  C ||h||_\infty \int_{0}^{2}
\argsinh\left( \frac { \sqrt{u}}{|1-u|} \right)u^{-\frac 1 2 }  du \leq C ||h||_\infty.\]
and the second term can be written as  
\begin{align*}
I_2 &  = - \frac1{2\pi}   \int_{2}^{\frac{1}{1-x^2}}
   2 x\sqrt{1-(1-x^2)u}
    h(\sqrt{1-(1-x^2)u}) u^{-1} du +R\\
   & = - \frac1{2\pi} \int_{2(1-x^2)}^{1}
   2 x
   \sqrt{1-v}h(\sqrt{1-v}) v^{-1}  dv +R
\end{align*}
where
\begin{align*}
R & \leq C||h||_\infty  \int_{2}^{\frac{1}{1-x^2}}
u^{-2} du\\
&\leq C.
\end{align*}

Finally, we write
\begin{align*}
- \frac1{2\pi} \int_{2(1-x^2)}^{1}
   &2 x
   \sqrt{1-v}h(\sqrt{1-v}) v^{-1}  dv  \\ 
   & = - \frac1{2\pi} \int_{2(1-x^2)}^{1}
   2 
   \sqrt{1-v}h(\sqrt{1-v}) v^{-1}  dv 
    \quad  \\
& \quad + \frac1{2\pi} \int_{2(1-x^2)}^{1}
   2 (1-x)
   \sqrt{1-v}h(\sqrt{1-v}) v^{-1}  dv     
\end{align*}
where the second term is bounded as $x\to 1$ (because
$\frac{(1-x)}{v}\leq C$ for $v\geq 2 (1-x^2)$), and
\begin{align*}
- \frac1{2\pi} \int_{2(1-x^2)}^{1} &  2    \sqrt{1-v}h(\sqrt{1-v})
v^{-1}  dv \\
& = - \frac1{2\pi} \int_{2(1-x^2)}^{1}   2   h(\sqrt{1-v}) v^{-1}  dv\\
& \quad +  \frac1{2\pi} \int_{2(1-x^2)}^{1}   2 (1-   \sqrt{1-v})h(\sqrt{1-v}) v^{-1}  dv
\end{align*}
where the second term is again bounded as $x\to 1$ (because $ \frac{ (1-   \sqrt{1-v})}{v}\leq C$).
The lemma follows.
\end{proof}

\section{Derivation of the pressure law}
\label{app:pressure}

We recall here the main step of the derivation of the pressure law from linear elasticity equations
in the particular geometry of a crack of plain strain. 
These can computations can be found elsewhere (\cite{CS83,Peirce}) and are recall here for the reader's sake.

\subsection{Linear Elasticity equations}

The \textbf{strain tensor} $\boldsymbol \epsilon$ 
is related to the \textbf{displacement} $\mathbf{u}$ through the following equality
\begin{equation} \label{eq:strain}
\boldsymbol\epsilon = \frac12 ( \nabla \mathbf{u} + (\nabla \mathbf{u})^T) 
\end{equation}
where $\nabla \mathbf{u}$ denotes the Jacobian matrix of $u$. The \textbf{stress tensor} is denoted by $\boldsymbol\sigma$. 

We next recall the equations of linear elasticity.
\begin{itemize}
\item {\bf Force equilibrium} considerations show that the components of the stress tensor must satisfy the equations
  \[ \divv \boldsymbol\sigma + \mathbf{F} = 0 \] 
where $\mathbf{F}$ denotes body forces (such as gravity).

\item The {\bf stress-strain} relations for an isotropic linearly elastic material can be written in the form:
\begin{equation}\label{eq:s-s}
\boldsymbol\epsilon = \frac1E (\boldsymbol\sigma - \nu [\trace (\boldsymbol\sigma) I - \boldsymbol\sigma ]) 
\end{equation}
where $E$ is Young's modulus and $\nu$ is Poisson's ratio.
\end{itemize}

\subsection{2D plane-strain problems}

\begin{itemize}
\item The components of the symmetric 2-tensor $\boldsymbol\sigma$ are denoted by
  $\sigma_{xx}$, $\sigma_{yy}$, $\sigma_{zz}$,
  $\sigma_{xy}=\sigma_{yx}$, $\sigma_{xz}=\sigma_{zx}$ and
  $\sigma_{yz}=\sigma_{zy}$.
\item The  components of the vector field $\textbf{u}$ are denoted by $u_x$, $u_y$
  and $u_z$.
\item The components of the vector field $\boldsymbol\epsilon$ are denoted by  $\epsilon_{xx}$, $\epsilon_{yy}$, $\epsilon_{zz}$,
  $\epsilon_{xy}=\epsilon_{yx}$, $\epsilon_{xz}=\epsilon_{zx}$ and
  $\epsilon_{yz}=\epsilon_{zy}$.
\end{itemize}

If the solid is in a state of plain strain (parallel to the $xy$
plane), then $u_z=0$ and the components $u_x$ and $u_y$ of the
displacement are independent of the $z$ coordinate. As a consequence,
the strain tensor components $\epsilon_{zz}$, $\epsilon_{xz}=\epsilon_{zx}$ and
$\epsilon_{yz}=\epsilon_{zy}$ are zero, and the remaining components are independent
of $z$.

We note that the three remaining strain components are defined in
terms of two displacements.  This implies that they cannot be
specified independently. In fact, we can easily verify that if the
displacement are continuously differentiable, then the strain tensor
components must satisfy the following {\bf compatibility condition}
\begin{equation}\label{eq:comp}
\frac{\pa^2 \epsilon_{xx}}{\pa y^2}+\frac{\pa^2 \epsilon_{yy}}{\pa x^2} = 2 \frac{\pa^2 \epsilon_{xy}}{\pa x\pa y}.
\end{equation}

Furthermore, the third equation in (\ref{eq:s-s}) implies
$$ \sigma_{zz}= \nu (\sigma_{xx}+\sigma_{yy})$$
and the last two equations give $ \sigma_{xz} =\sigma_{yz} =0$.

The stress-strain relations (\ref{eq:s-s}) can thus be rewritten as:
\begin{equation}\label{eq:s-s2}
\left\{
\begin{array}{l}
\epsilon_{xx} =\frac{1}{2G}[\sigma_{xx}-\nu (\sigma_{xx}+\sigma_{yy})] \\[5pt]
\epsilon_{yy} =\frac{1}{2G} [\sigma_{yy}-\nu  (\sigma_{xx}+\sigma_{yy})]\\[5pt]
\epsilon_{xy} = \frac{1}{2G} \sigma_{xy} 
\end{array}
\right.
\end{equation}
where  $G=\frac{1}{2}\frac{E}{1+\nu}$ is the shear modulus
and the equilibrium conditions (without body forces inside the solid,
$F\equiv 0$) yield
\begin{equation}\label{eq:equi}
\left\{
\begin{array}{l}
\frac{\pa\sigma_{xx}}{\pa x} + \frac{\pa\sigma_{xy}}{\pa y} =0\\[5pt]
\frac{\pa\sigma_{yy}}{\pa y}+ \frac{\pa\sigma_{xy}}{\pa x}  =0.
\end{array}
\right.
\end{equation}

At this point, we note that (\ref{eq:s-s2}) and (\ref{eq:equi})
provide $5$ equations with $5$ unknowns ($\sigma_{xx}$, $\sigma_{yy}$,
$\sigma_{xy}$, $u_x$, $u_y$). 

\vspace{10pt}

\paragraph{The Airy stress function.} 
The equilibrium equations \eqref{eq:equi} imply the existence of a
function $U(x,y)$ (the Airy stress function) such that the three
components of the stress tensor can be written as
$$ \sigma_{xx} = \frac{\pa^2 U}{\pa {y}^2}, \quad \sigma_{yy} = \frac{\pa^2 U}{\pa{x}^2}, \quad \sigma_{xy} = -\frac{\pa^2 U}{\pa x\pa y}.$$
Furthermore,  the compatibility condition \eqref{eq:comp} and equation \eqref{eq:s-s2} imply
$$
\frac{\pa^2 \sigma_{xx}}{\pa y^2}+\frac{\pa^2 \sigma_{yy}}{\pa x^2} = 2 \frac{\pa^2 \sigma_{xy}}{\pa x\pa y}
$$
which yields:
$$ \Delta^2 U=0$$
(so $U$ is biharmonic).

\vspace{10pt}

Recalling that $\epsilon_{xx}$, $\epsilon_{yy}$ and $\epsilon_{xy}$ are defined in terms of
the displacements $u_x$ and $u_y$ by \eqref{eq:strain}, we finally
rewrite \eqref{eq:s-s2} as follows:
\begin{equation}\label{eq:U}
\left\{
\begin{array}{l}
\ds \frac{\pa u_x}{\pa x} =\frac{1}{2G}\left[\frac{\pa^2 U}{\pa {y}^2}-\nu \left(\frac{\pa^2 U}{\pa {y}^2}+\frac{\pa^2 U}{\pa{x}^2}\right)\right] \\[12pt]
\ds\frac{\pa u_y}{\pa y} =\frac{1}{2G} \left[\frac{\pa^2 U}{\pa{x}^2}-\nu  \left(\frac{\pa^2 U}{\pa {y}^2}+\frac{\pa^2 U}{\pa{x}^2}\right)\right]\\[12pt]
\ds\frac{\pa u_x}{\pa y}+\frac{\pa u_y}{\pa x} = -\frac{1}{G} \frac{\pa^2 U}{\pa x\pa y}.
\end{array}
\right.
\end{equation}
We now have reduced the problem to finding a biharmonic potential
$U(x,y)$ and the displacements $u_x(x,y)$, $u_y(x,y)$ such that
(\ref{eq:U}) holds (together with some appropriate boundary
conditions).

\subsection{Derivation of the pressure law for a $2$-D crack on an infinite domain}

We consider a fracture of opening $w$ in an infinite solid occupying the whole space $\RR^2$.
The fracture is assumed to be symmetric with respect to the $y=0$ axis, so that we only need to consider the problem in the upper half $\{y>0\}$.
Along $y=0$, we have the following boundary conditions:
$$
\sigma_{xy}(x,0)=0 \; \mbox{ and } \;u_y(x,y) = \frac{1}{2} w(x) \quad \mbox{ for all $x\in \RR$}
$$
and we assume 
$$ \sigma_{ij}\longrightarrow  0 \mbox{ as } |(x,y)|\to \infty.$$
Our goal is to determine the pressure
$$ p(x)=-\sigma_{yy}(x,0).$$
The main result of this section is the following: 
\begin{theorem}
The pressure $p(x)$ satisfies
$$ p(x)= \frac{E}{4(1-\nu^2)} (-\Delta)^{1/2} w(x) \quad \mbox{ for } \quad x\in \RR.$$
\end{theorem}
\begin{proof}
We use the Fourier transform with respect to $x$.
Denoting 
$$\hU(k,y)=\int_\R U(x,y) e^{-ikx}\, dx,$$
the biharmonic equation yields
$$ \left(\frac{d^2}{dy^2} -k^2\right)^2 \hU(k,y)=0$$
and so (using the conditions as $|y|\to \infty$)
$$ \hU(k,y) = (A(k)+B(k)y)e^{-|k|y} \quad \mbox{ for all $k\in\R$, $y>0$.} $$
Next, Equation (\ref{eq:U}) implies
\begin{equation}\label{eq:UF}
\left\{
\begin{array}{l}
-ik\hux(k,y)  =\frac{1}{2G}[(1-\nu) \frac{\pa ^2 \hU}{\pa y^2}(k,y) +\nu k^2 \hU(k,y)] \\[5pt]
\frac{\pa \huy}{\pa y}(k,y) =\frac{1}{2G} [-k^2(1-\nu)\hU(k,y)  -\nu  \frac{\pa^2 \hU}{\pa {y}^2}]\\[5pt]
\frac{\pa \hux}{\pa y}(k,y)-ik\huy(k,y) = \frac{ik}{G}  \frac{\pa \hU}{\pa y}(k,y).
\end{array}
\right.
\end{equation}
The first equation yields
$$
-i\hux  =\frac{1}{2G}\left[(1-\nu)\frac{1}{k} \frac{\pa ^2 \hU}{\pa y^2} +\nu k \hU\right] 
$$
and the last equation then implies
\begin{eqnarray*}
\huy(k,y) & = & -\frac{i}{k}\frac{\pa \hux}{\pa y}-\frac{1}{G}  \frac{\pa \hU}{\pa y} \\
& =& \frac{1}{2G}\left[(1-\nu)\frac{1}{k^2} \frac{\pa ^3 \hU}{\pa y^3} +(\nu-2)  \frac{\pa \hU}{\pa y}  \right]. 
\end{eqnarray*}

A simple computation gives
\begin{eqnarray*}
 \frac{\pa \hU}{\pa y}(k,y) &=&(B(k)-|k|A(k)-|k|B(k)y)e^{-|k|y}\\
 \frac{\pa^2 \hU}{\pa y^2}(k,y) &=& (-2|k|B(k) +k^2A(k)+k^2B(k)y)e^{-|k|y}\\
 \frac{\pa^3 \hU}{\pa y^3}(k,y) &=& (3k^2B(k) -|k|^3A(k)-|k|^3B(k)y)e^{-|k|y} .
\end{eqnarray*}
We recall that $\sigma_{xy}=-\frac{\pa^2 U}{\pa x\pa y}$ and so 
$$ \widehat{\sigma_{xy}} = ik\frac{\pa \hU(k,y)}{\pa y}.$$
In particular, 
the condition $\sigma_{xy}(x,0)=0$ for all $x\in \Omega$ implies $\frac{\pa \hU}{\pa y}(k,0)=0$ for all $k\in \N$ and so 
$$ B(k)=|k|A(k)\quad \mbox{ for all $k\in\R$.}$$
The condition $u_y=\frac{1}{2}w$ then gives
$$ 
\frac{1}{2G} (1-\nu)\frac{1}{k^2} \frac{\pa ^3 \hU}{\pa y^3}  (k,0)= \frac{1}{2}\widehat w(k)\quad \mbox{ for all $k\in\R$,}$$
which implies
$$\frac{2(1-\nu)}{G} |k|A(k) =\widehat w(k) \quad \mbox{ for all $k\in\R$.}$$
We deduce
$$ 
\hU(k,0) =A= \frac{G}{2(1-\nu)} \frac{1}{|k|} \widehat w(k)=\frac{E}{4(1-\nu^2)} \frac{1}{|k|} \widehat w(k)\quad \mbox{ for all $k\in\R$},$$ 
and so
$$ 
\widehat p(k)=-\widehat{\sigma_{yy}}(k,0) = k^2\widehat U(k,0) = \frac{E}{4(1-\nu^2)} |k| \widehat w(k)\quad \mbox{ for all $k\in\R$},$$
which is the Fourier transform of the equation
$$ p(x)= \frac{E}{4(1-\nu^2)} (-\Delta)^{1/2} w(x) \mbox{ for $x\in \RR$}. \qedhere$$
\end{proof}

\section{Proof of Lemma \ref{lem:derfb}}
\label{app:bb}

In this section, we give the proof of Lemma \ref{lem:derfb}, which relates the behavior of $u$ and $p$ at the tip of the fracture.
For that purpose, we rewrite \eqref{eq:pressure0} as
\begin{equation}\label{eq:pressure1}
-I(u)=\frac{4(1-\nu^2)}{E} p(x).
\end{equation}

\begin{proof}[Proof of Lemma~\ref{lem:derfb}]
We first prove \eqref{eq:asymp2}. 
For that we use \eqref{eq:pressure1} and Lemma \eqref{lem:greenG} to write
\begin{equation}\label{eq:upressure}
u(x)=\frac{4(1-\nu^2)}{E} \int_{-1}^1 G(x,y) p(y)\, dy \qquad \mbox{ for } x\in(-1,1)
\end{equation}
and so using Lemma \ref{lem:derivatives-G}, we get
$$
u'(x)=
\frac{4(1-\nu^2)}{\pi E}  \int_{-1}^1 
\frac{\sqrt{1-y^2}}{\sqrt{1-x^2}} \frac{p(y)}{y-x}\, dy.
$$
We deduce
$$ u'(x)\sqrt{1-x} = \frac{4(1-\nu^2)}{\pi E} \frac{1}{\sqrt{1+x}}  \int_{-1}^1 
\frac{\sqrt{1-y^2}}{y-x} p(y)\, dy
$$
hence
\begin{align*} 
\lim_{x\to 1^- }u'(x)\sqrt{1-x} & = \frac{4(1-\nu^2)}{\pi E} \frac{1}{\sqrt{2}}  \int_{-1}^1 
\frac{\sqrt{1-y^2}}{y-1} p(y)\, dy \\
& = - \frac{4(1-\nu^2)}{\pi E} \frac{1}{\sqrt{2}}  \int_{-1}^1 
\frac{\sqrt{1+y}}{\sqrt{1-y}} p(y)\, dy
\end{align*}
which is \eqref{eq:asymp2}.

\medskip

We now turn to the proof of \eqref{eq:asymp1}. First, we recall that the square root of the Laplacian can also be represented by a singular integral:
$$(-\Delta)^{1/2}(u)=\frac{1}{\pi} \mbox{P.V.} \int_\RR \frac{u(x)-u(y)}{|x-y|^2}\, dy.$$
In view of the pressure law \eqref{eq:pressure0}, we deduce:
$$ p(x)= \frac{E}{4(1-\nu^2)} (-\Delta)^{1/2} u =   \frac{E}{4(1-\nu^2)} \frac 1 \pi \mbox{P.V.} \int_{-\infty}^{+\infty } \frac{u(x)-u(y)}{|x-y|^2}\, dy$$
for all $x\in \RR$. In particular, using the fact that $\mbox{supp}\, u = (-1,1)$, we deduce that for $x>1$, we have 
$$ p(x)= - \frac{E}{4(1-\nu^2)} \frac 1 \pi \int_{-1}^{1} \frac{u(y)}{|x-y|^2}\, dy\qquad \mbox{ for all $x>1$}$$
(note that the principal value is no longer necessary here).
Using \eqref{eq:upressure} in this last expression, we get
$$ p(x)= -   \frac 1 \pi \int_{-1}^{1} \frac{1}{|x-y|^2} \left\{ \int_{-1}^1 G(y,z) p(z)\, dz \right\}\, dy\qquad \mbox{ for all $x>1$}$$
and so
$$ -\sqrt{x-1}\, p(x) =  \frac 1 \pi \int_{-1}^{1} \int_{-1}^1 \frac{\sqrt{x-1}G(y,z) }{|x-y|^2} p(z)\, dy\, dz\qquad \mbox{ for all $x>1$}.$$
Using the change of variable $t=\frac{1-y}{x-1}$ we have
\begin{align*}
\int_{-1}^{1} \frac{\sqrt{x-1}G(y,z) }{|x-y|^2}  \, dy& =  \int_{0}^{\frac{2}{x-1}} \frac{\sqrt{x-1}G(1-(x-1)t,z) }{(x-1)^2 |1+t|^2}  \, (x-1) dt\\
& =  \int_{0}^{\frac{2}{x-1}} \frac{G(1-(x-1)t,z) }{\sqrt{x-1}}  \, \frac{dt}{ |1+t|^2}
\end{align*}
and  formula \eqref{eq:G1} implies that for  $y\in(-1,1)$, $z\in(-1,1)$,
$$\lim_{x\to1^+} \frac{G(1-(x-1)t,z) }{\sqrt{x-1}}= \frac{1}{\pi} \frac{\sqrt{2t}\sqrt{1-z^2}}{|1-z|}= \frac{1}{\pi} \frac{\sqrt{2t}\sqrt{1+z}}{\sqrt{1-z}}.$$
We deduce (arguing as in Section~3 to justify exchanging limits and integrals)
\begin{align*}
\lim_{x\to 1^+}  -\sqrt{x-1}p(x)& =  \frac 1 \pi \int_{-1}^{1} \int_0^\infty  \frac{1}{\pi} \frac{\sqrt{2t}\sqrt{1+z}}{\sqrt{1-z}}\frac{dt}{ |1+t|^2} p(z)\, dz \\
& =   \frac{1}{\pi^2} \int_{-1}^{1} \int_0^\infty  \frac{\sqrt{2t}}{|1+t|^2}dt \frac{\sqrt{1+z}}{\sqrt{1-z}}p(z)\, dz.
\end{align*}
The result now follows using the fact that  $\int_0^\infty  \frac{\sqrt{2t}}{|1+t|^2}dt=\frac{\pi}{\sqrt 2}$.
\end{proof}

\paragraph{Acknowledgements.} C.I. is partially supported by projects
IDEE ANR-2010-0112-01 and HJnet ANR-12-BS01-0008-01. A.M. is partially supported by NSF Grant DMS-1201426. Part of this work was completed while A.M. was holding the Junior Chair of the Fondation Sciences Math\'ematiques de Paris.

\bibliographystyle{siam}
\bibliography{crack,self}

\end{document}